\documentclass[10pt]{article}
\usepackage{amsmath,amsthm}
\usepackage{amssymb}
\usepackage{hyperref}

\begin{document}
\title{Convolution identities for Tetranacci numbers}
\author{
Rusen Li \\
\small School of Mathematics and Statistics\\
\small Wuhan University\\
\small Wuhan 430072 China\\
\small \texttt{limanjiashe@whu.edu.cn}
}
%\thanks{
%This research was supported in part by the grant of Wuhan University and by the grant of Hubei Provincial Experts Program.
%}
\date{
%\small Submitted: November 1, 2016;  Accepted: December 2, 2016.\\
%\small MR Subject Classifications: Primary 11B39; Secondary 11B37, 05A15, 05A19
}

\maketitle

\def\stf#1#2{\left[#1\atop#2\right]}
\def\sts#1#2{\left\{#1\atop#2\right\}}
\def\fl#1{\left\lfloor#1\right\rfloor}
\def\cl#1{\left\lceil#1\right\rceil}

\newtheorem{theorem}{Theorem}
\newtheorem{Prop}{Proposition}
\newtheorem{Cor}{Corollary}
\newtheorem{Lem}{Lemma}

\begin{abstract}
We give convolution identities without binomial coefficients for Tetranacci numbers and convolution identities with binomial coefficients for Tetranacci and Tetranacci-type numbers.
\end{abstract}

\section{Introduction}

\quad Convolution identities for various types of numbers (or polynomials) have been studied, with or without
binomial coefficients, including Bernoulli, Euler, Genocchi, Catalan, Cauchy, Stirling, Fibonacci and Tribonacci numbers (\cite{AD1,AD2,AD3,KK,Sunzhiwei,Komatsu2015,Komatsu2016,KS2016,KL,KMP,Kiric2008}).Tetranacci sequence has been studied in \cite{WM1,WM2,GL}.
% One typical formula is due to Euler, given by
%$$
%\sum_{k=0}^n\binom{n}{k}\mathcal B_k\mathcal B_{n-k}=-n\mathcal B_{n-1}-(n-1)\mathcal B_n\quad(n\ge 0)\,,
%$$
%where $\mathcal B_n$ are Bernoulli numbers, defined by
%$$
%\frac{x}{e^x-1}=\sum_{n=0}^\infty\mathcal B_n\frac{x^n}{n!}\quad(|x|<2\pi)\,.
%$$

%In \cite{KR}, the convolution identities for general Fibonacci-type numbers $u_n$, satisfying the recurrence relations $u_n=a u_{n-1}+b u_{n-2}$, have been studied.  By considering the roots of the quadratic equation $x^2-a x-b=0$,
%$$
%\sum_{k_1+\cdots+k_r=n\atop k_1,\dots,k_r\ge 0}\binom{n}{k_1,\dots,k_r}u_{k_1}\cdots u_{k_r}
%$$
%can be expressed in the linear combination of $u_1$, $\dots$, $u_n$, where
%$$
%\binom{n}{k_1,\dots,k_r}=\frac{n!}{k_1!\cdots k_r!}
%$$
%denotes the multinomial coefficient.  The case without binomial (multinomial) coefficients is discussed in \cite{KMP}, but in simpler Fibonacci numbers.

%However, the situation becomes more complicated for the numbers related to a higher-order equation.  %Nevertheless,  the numbers satisfying the five-term recurrence relation have the possibility to have the convolution identities.
%as the quartic function can be solvable,
{\it Tetranacci numbers} $T_n$ are defined by the recurrence relation
\begin{equation}
T_n=T_{n-1}+T_{n-2}+T_{n-3}+T_{n-4}\quad(n\ge 4)\quad\hbox{with}\quad T_0=0,~T_1=T_2=1,~T_3=2
\label{def:tetra}
\end{equation}
and their sequence is given by
$$
\{T_n\}_{n\ge 0}= 0, 1, 1, 2, 4, 8, 15, 29, 56, 108, 208, \dots
$$
(\cite[A000078]{oeis}).

The generating function without factorials is given by
\begin{equation}
T(x):=\frac{x}{1-x-x^2-x^3-x^4}=\sum_{n=0}^\infty T_n x^n
\label{gen:nofacto}
\end{equation}
because of the recurrence relation (\ref{def:tetra}).

On the other hand, the generating function with binomial coefficients is given by
\begin{equation}
t(x):=c_1 e^{\alpha x}+c_2 e^{\beta x}+c_3 e^{\gamma x}+c_4 e^{\delta x}=\sum_{n=0}^\infty T_n\frac{x^n}{n!}\,,
\label{gen:facto}
\end{equation}
where $\alpha$, $\beta$, $\gamma$ and $\delta$ are the roots of $x^4-x^3-x^2-x-1=0$
and
\begin{align*}
c_1:&=\frac{2-(\beta+\gamma+\delta)+(\beta\gamma+\gamma\delta+\delta\beta)}
{(\alpha-\beta)(\alpha-\gamma)(\alpha-\delta)}\\
&=\frac{1}{-\alpha^3+6\alpha-1}\,,\\
c_2:&=\frac{2-(\alpha+\gamma+\delta)+(\alpha\gamma+\gamma\delta+\delta\alpha)}
{(\beta-\alpha)(\beta-\gamma)(\beta-\delta)}\\
&=\frac{1}{-\beta^3+6\beta-1}\,,\\
c_3:&=\frac{2-(\alpha+\beta+\delta)+(\alpha\beta+\beta\delta+\delta\alpha)}
{(\gamma-\alpha)(\gamma-\beta)(\gamma-\delta)}\\
&=\frac{1}{-\gamma^3+6\gamma-1}\,,\\
c_4:&=\frac{2-(\alpha+\beta+\gamma)+(\alpha\beta+\beta\gamma+\gamma\alpha)}
{(\delta-\alpha)(\delta-\beta)(\delta-\gamma)}\\
&=\frac{1}{-\delta^3+6\delta-1}\,.
\end{align*}

Notice that
\begin{align*}
c_1+c_2+c_3+c_4&=0\,,\\
c_1\alpha+c_2\beta+c_3\gamma+c_4\delta&=1\,,\\
c_1\alpha^2+c_2\beta^2+c_3\gamma^2+c_4\delta^2&=1\,,\\
c_1\alpha^3+c_2\beta^3+c_3\gamma^3+c_4\delta^3&=2,
\end{align*}
because $t_n$ has a Binet-type formula:
$$
T_n=c_1\alpha^n+c_2\beta^n+c_3\gamma^n+c_4\gamma^n\quad(n\ge 0)\,.
$$

%Convolution identities for various types of numbers (or polynomials) have been studied, with or without binomial (or multinomial) coefficients, including Bernoulli, Euler, Genocchi, Cauchy, Stirling, and Fibonacci numbers (\cite{AD1,AD2,AD3,Komatsu2015,Komatsu2016,KMP,KS2016}). One typical formula is due to Euler, given by
%$$
%\sum_{k=0}^n\binom{n}{k}\mathcal B_k\mathcal B_{n-k}=-n\mathcal B_{n-1}-(n-1)\mathcal B_n\quad(n\ge 0)\,,
%$$
%where $\mathcal B_n$ are Bernoulli numbers, defined by
%$$
%\frac{x}{e^x-1}=\sum_{n=0}^\infty\mathcal B_n\frac{x^n}{n!}\quad(|x|<2\pi)\,.
%$$

%In this paper, we give convolution identities without binomial coefficients by using the generating function (\ref{gen:nofacto}), and those with binomial coefficients by using the generating function (\ref{gen:facto}).
In this paper, we give convolution identities without binomial coefficients for Tetranacci numbers and convolution identities with binomial coefficients for Tetranacci and Tetranacci-type numbers.

\section{Convolution identities without binomial coefficients}

By (\ref{gen:nofacto}), we have
$$
T'(x)=\frac{1+x^2+2 x^3+3 x^4}{(1-x-x^2-x^3-x^4)^2}\,.
$$
Hence,
\begin{equation}
(1+x^2+2 x^3+3 x^4)T(x)^2=x^2 T'(x)\,.
\label{eq:201}
\end{equation}
The left-hand side of (\ref{eq:201}) is
\begin{align*}
&(1+x^2+2 x^3+3 x^4)\sum_{n=0}^\infty\sum_{k=0}^n T_k T_{n-k}x^n\\
&=\sum_{n=0}^\infty\sum_{k=0}^n T_k T_{n-k}x^n
+\sum_{n=2}^\infty\sum_{k=0}^{n-2}T_k T_{n-k-2}x^n\\
&\qquad
+2\sum_{n=3}^\infty\sum_{k=0}^{n-3}T_k T_{n-k-3}x^n
+3\sum_{n=4}^\infty\sum_{k=0}^{n-4}T_k T_{n-k-4}x^n\\
&=\sum_{n=4}^\infty\sum_{k=0}^{n-4}t_k(T_{n-k}+T_{n-k-2}+2 T_{n-k-3}+3 T_{n-k-4})x^n\\
&\qquad +\sum_{n=4}^\infty(T_{n-1}+T_{n-2}+3 T_{n-3})x^n+x^2+2 x^3\,.
\end{align*}
The right-hand side of (\ref{eq:201}) is
\begin{align*}
x^2\sum_{n=0}^\infty(n+1)T_{n+1}x^n
=\sum_{n=2}^\infty(n-1)T_{n-1}x^n\,.
\end{align*}
Therefore, we get the following result.

\begin{theorem}
For $n\ge 4$, we have
$$
\sum_{k=0}^{n-4}T_k(T_{n-k}+T_{n-k-2}+2 T_{n-k-3}+3 T_{n-k-4})
=(n-2)T_{n-1}-T_{n-2}-3 T_{n-3}\,.
$$
\label{th:nobinom}
\end{theorem}

The identity (\ref{eq:201}) can be written as
\begin{equation}
T(x)^2=\frac{x^2}{1+x^2+2 x^3+3 x^4}T'(x)\,.
\label{eq:202}
\end{equation}
Since
\begin{align*}
\frac{1}{1+x^2+2 x^3+3 x^4}
&=\sum_{l=0}^\infty(-1)^l x^{2 l}(1+2x+3x^2)^l\\
&=\sum_{l=0}^\infty(-1)^l x^{2 l}\sum_{i+j+k=l\atop i,j,k\ge0}\binom{l}{i,j,k}1^i(2 x)^j (3x^2)^k\\
&=\sum_{m=0}^\infty\sum_{j,k=0}^{3j+2k\leq m \atop j+4k\leq m}(-1)^{\frac{m-j-2k}{2}}\frac{1+(-1)^{m-j-2k}}{2}\\
&\qquad \times\binom{\frac{m-j-2k}{2}}{\frac{m-j-2k}{2}-j-k,j,k}2^j 3^k x^m,
\end{align*}
and
$$
T'(x)=\sum_{n=0}^\infty(n+1)T_{n+1}x^n\,,
$$
the right-hand side of (\ref{eq:202}) is
\begin{align*}
x^2A\sum_{l=0}^\infty (l+1)T_{l+1}x^l
=x^2\sum_{n=0}^\infty \sum_{l=0}^n B(l+1)T_{l+1} x^n\\
=\sum_{n=2}^\infty \sum_{l=0}^{n-2} C(l+1)T_{l+1} x^n\,,
\end{align*}
where
\begin{align*}
&A=\sum_{m=0}^\infty\sum_{j,k=0}^{3j+2k\leq m \atop j+4k\leq m}(-1)^{\frac{m-j-2k}{2}}\frac{1+(-1)^{m-j-2k}}{2}\binom{\frac{m-j-2k}{2}}{\frac{m-j-2k}{2}-j-k,j,k}2^j 3^k x^m,\\
&B=\sum_{j,k=0}^{3j+2k\leq n-l \atop j+4k\leq n-l}(-1)^{\frac{n-l-j-2k}{2}}\frac{1+(-1)^{n-l-j-2k}}{2}\binom{\frac{n-l-j-2k}{2}}{\frac{n-l-3j-4k}{2},j,k}2^j 3^k,\\
&C=\sum_{j,k=0}^{3j+2k\leq n-2-l \atop j+4k\leq n-2-l}(-1)^{\frac{n-2-l-j-2k}{2}}\frac{1+(-1)^{n-2-l-j-2k}}{2}\binom{\frac{n-2-l-j-2k}{2}}{\frac{n-2-l-3j-4k}{2},j,k}2^j 3^k.
\end{align*}
Since the left-hand side of (\ref{eq:202}) is
$$
\sum_{n=0}^\infty\sum_{k=0}^n T_k T_{n-k}x^n\,,
$$
comparing the coefficients on both sides, we obtain the following result without binomial coefficient.

\begin{theorem}
For $n\ge 2$,
\begin{align*}
\sum_{k=0}^n T_k T_{n-k}=\sum_{l=0}^{n-2}(l+1)T_{l+1}D,
\end{align*}
where $$ D=\sum_{j,k=0}^{3j+2k\leq n-2-l \atop j+4k\leq n-2-l}(-1)^{\frac{n-2-l-j-2k}{2}}\frac{1+(-1)^{n-2-l-j-2k}}{2}
\binom{\frac{n-2-l-j-2k}{2}}{\frac{n-2-l-3j-4k}{2},j,k}2^j 3^k. $$
\end{theorem}

\section{Some preliminary lemmas}

For convenience, we shall introduce modified Tetranacci numbers $T_n^{(s_0,s_1,s_2,s_3)}$, satisfying the recurrence relation
$$
T_n^{(s_0,s_1,s_2,s_3)}=T_{n-1}^{(s_0,s_1,s_2,s_3)}+T_{n-2}^{(s_0,s_1,s_2,s_3)}
+T_{n-3}^{(s_0,s_1,s_2,s_3)}+T_{n-4}^{(s_0,s_1,s_2,s_3)}\quad(n\ge 4)
$$
with given initial values $T_0^{(s_0,s_1,s_2,s_3)}=s_0$, $T_1^{(s_0,s_1,s_2,s_3)}=s_1$, $T_2^{(s_0,s_1,s_2,s_3)}=s_2$,and $T_3^{(s_0,s_1,s_2,s_3)}=s_3$.  Hence, $T_n=T_n^{(0,1,1,2)}$ are ordinary Tetranacci numbers.

First, we shall prove the following four lemmata.

\begin{Lem}
We have
$$
c_1^2 e^{\alpha x}+c_2^2 e^{\beta x}+c_3^2 e^{\gamma x}+ c_4^2 e^{\delta x}=\frac{1}{563}\sum_{n=0}^\infty T_n^{(40,64,215,344)}\frac{x^n}{n!}\,.
$$
\label{c^2}
\end{Lem}
\begin{proof}
For Tetranacci-type numbers $s_n$, satisfying the recurrence relation $s_n=s_{n-1}+s_{n-2}+s_{n-3}+s_{n-4}$ ($n\ge 4$) with given initial values $s_0$, $s_1$, $s_2$ and $s_3$, we have
\begin{equation}
d_1 e^{\alpha x}+d_2 e^{\beta x}+d_3 e^{\gamma x}+d_4 e^{\delta x}=\sum_{n=0}^\infty s_n\frac{x^n}{n!}\,.
\label{gen:tetra-type}
\end{equation}
Since $d_1$, $d_2$, $d_3$ and $d_4$ satisfy the system of the equations
\begin{align*}
d_1+d_2+d_3+d_4&=s_0\,,\\
d_1\alpha+d_2\beta+d_3\gamma+d_4\gamma&=s_1\,,\\
d_1\alpha^2+d_2\beta^2+d_3\gamma^2+d_4\gamma^2&=s_2\,,\\
d_1\alpha^3+d_2\beta^3+d_3\gamma^3+d_4\gamma^3&=s_3\,,
\end{align*}
we have
\begin{align*}
d_1&=\dfrac{\left|
\begin{array}{cccc}
s_0&1&1&1\\
s_1&\beta&\gamma&\delta\\
s_2&\beta^2&\gamma^2&\delta^2\\
s_3&\beta^3&\gamma^3&\delta^3
\end{array}\right|}
{\left|\begin{array}{cccc}
1&1&1&1\\
\alpha &\beta&\gamma&\delta\\
\alpha^2 &\beta^2&\gamma^2&\delta^2\\
\alpha^3 &\beta^3&\gamma^3&\delta^3
\end{array}\right|}
=\frac{s_0\beta\gamma\delta+s_2(\beta+\gamma+\delta)-s_3-s_1(\beta\gamma+\beta\delta+\gamma\delta)}
{(\beta-\alpha)(\gamma-\alpha)(\delta-\alpha)}\,,
\end{align*}

\begin{align*}
d_2&=\dfrac{\left|
\begin{array}{cccc}
1&s_0&1&1\\
\alpha   &s_1&\gamma&\delta\\
\alpha^2 & s_2&\gamma^2&\delta^2\\
\alpha^3 &s_3&\gamma^3&\delta^3
\end{array}\right|}
{\left|\begin{array}{cccc}
1&1&1&1\\
\alpha &\beta&\gamma&\delta\\
\alpha^2 &\beta^2&\gamma^2&\delta^2\\
\alpha^3 &\beta^3&\gamma^3&\delta^3
\end{array}\right|}
=\frac{s_0\gamma\delta\alpha+s_2(\gamma+\delta+\alpha)-s_3-s_1(\gamma\delta+\gamma\alpha+\delta\alpha)}
{(\gamma-\beta)(\delta-\beta)(\alpha-\beta)}\,,
\end{align*}

\begin{align*}
d_3&=\dfrac{\left|
\begin{array}{cccc}
1&1&s_0&1\\
\alpha   &\beta   &s_1&\delta\\
\alpha^2 &\beta^2 &s_2&\delta^2\\
\alpha^3 &\beta^3 &s_3&\delta^3
\end{array}\right|}
{\left|\begin{array}{cccc}
1&1&1&1\\
\alpha &\beta&\gamma&\delta\\
\alpha^2 &\beta^2&\gamma^2&\delta^2\\
\alpha^3 &\beta^3&\gamma^3&\delta^3
\end{array}\right|}
=\frac{s_0\delta\alpha\beta+s_2(\delta+\alpha+\beta)-s_3-s_1(\delta\alpha+\delta\beta+\alpha\beta)}
{(\delta-\gamma)(\alpha-\gamma)(\beta-\gamma)}\,,
\end{align*}

\begin{align*}
d_4&=\dfrac{\left|
\begin{array}{cccc}
1&1&1&s_0\\
\alpha   &\beta   &\gamma&s_1\\
\alpha^2 &\beta^2 &\gamma^2&s_2\\
\alpha^3 &\beta^3 &\gamma^3&s_3
\end{array}\right|}
{\left|\begin{array}{cccc}
1&1&1&1\\
\alpha &\beta&\gamma&\delta\\
\alpha^2 &\beta^2&\gamma^2&\delta^2\\
\alpha^3 &\beta^3&\gamma^3&\delta^3
\end{array}\right|}
=\frac{s_0\alpha\beta\gamma+s_2(\alpha+\beta+\gamma)-s_3-s_1(\alpha\beta+\alpha\gamma+\beta\gamma)}
{(\alpha-\delta)(\beta-\delta)(\gamma-\delta)}\,.
\end{align*}
When $s_0=40$, $s_1=64$, $s_2=215$ and $s_3=344$,
by $\alpha+\beta+\gamma+\delta=1$, $\beta\gamma+\beta\delta+\gamma\delta=-1-(\alpha\beta+\alpha\gamma+\alpha\delta)=\alpha^2-\alpha-1$, $\alpha\beta\gamma\delta=1$ and $\alpha^4=\alpha^3+\alpha^2+\alpha+1$,
we have
$$
d_1=\frac{40\beta\gamma\delta+215(\beta+\gamma+\delta)-344-64(\beta\gamma+\beta\delta+\gamma\delta)}
{(\beta-\alpha)(\gamma-\alpha)(\delta-\alpha)}=563 c_1^2.
$$.
Similarly, we have $d_2=563 c_2^2$, $d_3=563 c_3^2$ and $d_4=563 c_4^2$.
\end{proof}

\begin{Lem}
We have
\begin{align*}
&\sum_{n=0}^\infty t_n \frac{x^n}{n!}=c_1 c_2 e^{(\alpha+\beta) x}+c_1 c_3 e^{(\alpha+\gamma) x}+c_1 c_4 e^{(\alpha+\delta)x}\\
& \qquad \qquad \quad +c_2 c_3 e^{(\beta+\gamma) x}+c_2 c_4 e^{(\beta+\delta) x}+c_3 c_4 e^{(\gamma+\delta)x}\,,
\end{align*}
\label{cc}
where
$$
t_n=\frac{1}{2}\left(\sum_{k=0}^n\binom{n}{k}T_k T_{n-k}-\frac{2^n}{563}T_n^{(40,64,215,344)}\right).
$$

\end{Lem}
\begin{proof}
Since
\begin{align*}
&(c_1 e^{\alpha x}+c_2 e^{\beta x}+c_3 e^{\gamma x}+ c_4 e^{\delta x})^2\\
&=c_1^2 e^{\alpha x}+c_2^2 e^{\beta x}+c_3^2 e^{\gamma x}+ c_4^2 e^{\delta x}+2(c_1 c_2 e^{(\alpha+\beta) x}+c_1 c_3 e^{(\alpha+\gamma) x}+c_1 c_4 e^{(\alpha+\delta)x}\\
& \quad +c_2 c_3 e^{(\beta+\gamma) x}+c_2 c_4 e^{(\beta+\delta) x}+c_3 c_4 e^{(\gamma+\delta)x}),
\end{align*}
we can obtain the following identity:
\begin{align*}
\left(\sum_{n=0}^\infty T_n \frac{x^n}{n!}\right)^2
&=\sum_{n=0}^\infty \sum_{k=0}^n\binom{n}{k}T_k T_{n-k}\frac{x^n}{n!}\\
&=\frac{1}{563}\sum_{n=0}^\infty T_n^{(40,64,215,344)}\frac{(2x)^n}{n!}
+2\sum_{n=0}^\infty t_n \frac{x^n}{n!}.
\end{align*}
Comparing the coefficients on both sides, we get the desired result.
\end{proof}

\begin{Lem}
We have
$$
c_2 c_3 c_4 e^{\alpha x}+ c_3c_4 c_1e^{\beta x}+c_4 c_1c_2 e^{\gamma x}+c_1 c_2 c_3 e^{\delta x}
=-\frac{1}{563}\sum_{n=0}^\infty T_n^{(-5,2,13,32)}\frac{x^n}{n!}\,.
$$
\label{ccc}
\end{Lem}

\begin{proof}
In the proof of Lemma \ref{c^2}, we put $s_0=-5$, $s_1=2$, $s_2=13$  and $s_3=32$, instead.
We have
$$
d_1=\frac{-5\beta\gamma\delta+13(\beta+\gamma+\delta)-32-2(\beta\gamma+\beta\delta+\gamma\delta)}
{(\beta-\alpha)(\gamma-\alpha)(\delta-\alpha)}=-563 c_2c_3c_4.
$$.
Similarly, we have $d_2=-563c_3c_4 c_1$,\quad $d_3=-563 c_4 c_1c_2$ and $d_4=-563 c_1 c_2 c_3$.
\end{proof}

%Now, let us consider the sum of three products with trinomial coefficients.

\begin{Lem}
We have
$$
c_1 c_2 c_3c_4=-\frac{1}{563}\,.
$$
\label{cccc}
\end{Lem}
\begin{proof}
By
 $\alpha+\beta+\gamma+\delta=1$, $\beta\gamma+\beta\delta+\gamma\delta=-1-(\alpha\beta+\alpha\gamma+\alpha\delta)=\alpha^2-\alpha-1$, $\alpha\beta\gamma\delta=1$ and $\alpha^4=\alpha^3+\alpha^2+\alpha+1$,
 we have
 \begin{align*}
 &c_1 c_2 c_3c_4\\
 &=\frac{\alpha^2}{(\alpha-\beta)(\alpha-\gamma)(\alpha-\delta)}
 \frac{\beta^2}{(\beta-\alpha)(\beta-\gamma)(\beta-\delta)}\\
 &\quad \times\frac{\gamma^2}{(\gamma-\alpha)(\gamma-\beta)(\gamma-\delta)}
 \frac{\delta^2}{(\delta-\alpha)(\delta-\beta)(\delta-\gamma)}\\
 &=\frac{\alpha^2 \beta^2 \gamma^2 \delta^2}{(\alpha-\beta)^2(\alpha-\gamma)^2(\alpha-\delta)^2
 (\beta-\gamma)^2(\gamma-\beta)^2(\beta-\delta)^2}\\
 &=\frac{1}{(4\alpha^3-3\alpha^2-2\alpha-1)^2(39\alpha^3-58\alpha^2-23\alpha-23)}\\
 &=-\frac{1}{563}.
\end{align*}
\end{proof}

\section{Convolution identities for three and four Tetranacci numbers}

   Before giving more convolution identities,we shall give some elementary algebraic identities in symmetric form.It is not so difficult to determine the relations among coefficients.

\begin{Lem}
\label{alg-3}
The following equality holds:
\begin{align*}
&(a+b+c+d)^3\\
&=A (a^3+b^3+c^3+d^3) + B (abc+abd+acd+bcd) \\
&\quad + C (a^2+b^2+c^2+d^2)(a+b+c+d) \\
& \quad + D(ab+ac+ad+bc+bd+cd)(a+b+c+d),
\end{align*}
where $A=D-2$,\quad $B=-3D+6$,\quad $C=-D+3$.
\end{Lem}

\begin{Lem}
\label{alg-4}
The following equality holds:
\begin{align*}
&(a+b+c+d)^4 \\
&=A(a^4+b^4+c^4+d^4)+B abcd+C(a^3+b^3+c^3+d^4)(a+b+c+d)\\
&\quad +D(a^2+b^2+c^2+d^2)^2+E(a^2+b^2+c^2+d^2)(ab+ac+ad+bc+bd+cd)\\
&\quad +F(ab+ac+ad+bc+bd+cd)^2+G(a^2+b^2+c^2+d^2)(a+b+c+d)^2\\
&\quad +H(ab+ac+ad+bc+bd+cd)(a+b+c)^2\\
&\quad +I(abc(a+b+c)+abd(a+b+d)+bcd(b+c+d)+acd(a+c+d))\\
&\quad +J(abc+abd+bcd+acd)(a+b+c+d),
\end{align*}
where $A=-D+E+G+H-3$, \quad $B=12D+12G-4J-12$,\quad \\
 $C=-E-2G-H+4$,\quad $F=-2D-2G-2H+6$,\quad $I=4D-E+2G-H-J$.
\end{Lem}
%+abd(ab+bd+ad)+acd(ac+ad+cd)

\begin{Lem}
\label{alg-5}
The following equality holds:
\begin{align*}
&(a+b+c+d)^5 \\
&=A(a^5+b^5+c^5+d^5)\\
&\quad +B(abc(ab+bc+ca)+abd(ab+bd+ad)+acd(ac+ad+cd) +bcd(bc+bd+cd))\\
&\quad +C(abc(a^2+b^2+c^2)+abd(b^2+c^2+d^2)\\
&\quad \quad +acd(a^2+c^2+d^2)+bcd(b^2+c^2+d^2))\\
&\quad +D(abc(a+b+c)^2+abd(a+b+d)^2+acd(a+c+d)^2+bcd(b+c+d)^2)\\
&\quad +E(a^4+b^4+c^4+d^4)(a+b+c+d)+F(a+b+c+d)abcd\\
&\quad +G(a+b+c+d)\\
&\quad \quad \times(abc(a+b+c)+abd(a+b+d)+bcd(b+c+d)+acd(a+c+d))\\
&\quad +H(a^3+b^3+c^3+d^3)(a^2+b^2+c^2+d^2)\\
\end{align*}
\begin{align*}
&\quad +I(a^3+b^3+c^3+d^3)(ab+ac+ad+bc+bd+cd)\\
&\quad +J(abc+abd+acd+bcd)(a^2+b^2+c^2+d^2)\\
&\quad +K(abc+abd+acd+bcd)(ab+ac+ad+bc+bd+cd)\\
&\quad +L(a^3+b^3+c^3+d^3)(a+b+c+d)^2\\
&\quad +M(abc+abd+acd+bcd)(a+b+c+d)^2\\
&\quad +N(a^2+b^2+c^2+d^2)^2(a+b+c+d)\\
&\quad +P(ab+ac+ad+bc+bd+cd)^2(a+b+c+d)\\
&\quad +Q(a^2+b^2+c^2+d^2)(ab+ac+ad+bc+bd+cd)(a+b+c+d)\\
&\quad +R(a^2+b^2+c^2+d^2)(a+b+c+d)^3\\
&\quad +S(ab+ac+ad+bc+bd+cd)(a+b+c+d)^3,
\end{align*}

where
\begin{align*}
&A=I+2L+2N+P+2Q+6R+4S-14,\\
 &B=-2D-2G-K-2M-2N-5P-2Q-6R-12S+30,\\
&C=-D-G-I-J-2L-M-2P-3Q-6R-7S+20,\\
 &E=-I-2L-N-Q-3R-S+5,\\
&F=-3G-J-3K-7M-12P-3Q-6R-27S+60,\\
& H=-L-2N-P-Q-4R-3S+10.
\end{align*}
\end{Lem}

\bigskip

Now, let us consider the sum of three products with trinomial coefficients.
\begin{Lem}
We have
$$
c_1^3 e^{\alpha x}+c_2^3 e^{\beta x}+c_3^3 e^{\gamma x}+c_4^3 e^{\delta x}
=\frac{1}{563}\sum_{n=0}^\infty T_n^{(15,27,48,107)}\frac{x^n}{n!}\,.
$$
\label{c^3}
\end{Lem}
\begin{proof}
In the proof of Lemma \ref{c^2}, we put $s_0=15$, $s_1=27$, $s_2=48$ and $s_3=107$, instead.
We can obtain that
$$
d_1=\frac{15\beta\gamma\delta+48(\beta+\gamma+\delta)-107-27(\beta\gamma+\beta\delta+\gamma\delta)}
{(\beta-\alpha)(\gamma-\alpha)(\delta-\alpha)}=563 c_1^3.
$$.
Similarly, we have $d_2=563c_2^3$, $d_3=563 c_3^3$ and $d_4=563 c_4^3$.
\end{proof}

By using Lemmata \ref{c^2}, \ref{cc}, \ref{ccc}, \ref{alg-3} and \ref{c^3}, we get the following result.
\begin{theorem}
For $n\ge 0$,
\begin{align*}
&\sum_{k_1+k_2+k_3=n\atop k_1,k_2,k_3\ge 0}\binom{n}{k_1,k_2,k_3}T_{k_1}T_{k_2}T_{k_3}\\
&=\frac{A}{563}3^nT_{n}^{(15,27,48,107)}-\frac{B}{563}\sum_{k=0}^n\binom{n}{k}T_{k}^{(-5,2,13,32)}(-1)^{k}\\
&\quad +\frac{C}{563}\sum_{k=0}^n\binom{n}{k}2^{n-k}T_{n-k}^{(40,64,215,344)}T_k
+D\sum_{k=0}^n\binom{n}{k}T_k t_{n-k}\,.
\end{align*}
\label{tri3}
where $A=D-2$, $B=-3D+6$, $C=-D+3$,
\begin{align*}
t_n=\frac{1}{2}\left(\sum_{k=0}^n\binom{n}{k}T_k T_{n-k}-\frac{2^n}{563}T_n^{(40,64,215,344)}\right).
\end{align*}
\end{theorem}

\noindent
{\it Remark.}
If we take $D=0$, we have for $n\ge 0$,
\begin{align*}
&\sum_{k_1+k_2+k_3=n\atop k_1,k_2,k_3\ge 0}\binom{n}{k_1,k_2,k_3}T_{k_1}T_{k_2}T_{k_3}\\
&=-\frac{2}{563}3^nT_{n}^{(15,27,48,107)}-\frac{6}{563}\sum_{k=0}^n\binom{n}{k}T_{k}^{(-5,2,13,32)}(-1)^{k}\\
&\quad +\frac{3}{563}\sum_{k=0}^n\binom{n}{k}2^{n-k}T_{n-k}^{(40,64,215,344)}T_k\,.
\end{align*}

\begin{proof}
First, by Lemmata\ref{c^2}, \ref{cc}, \ref{ccc}, \ref{alg-3} and \ref{c^3}, we have
\begin{align*}
&(c_1 e^{\alpha x}+c_2 e^{\beta x}+c_3 e^{\gamma x}+c_4 e^{\delta x})^3\\
&=A(c_1^3 e^{3\alpha x}+c_2^3 e^{3\beta x}+c_3^3 e^{3\gamma x}+c_4^3 e^{3\delta x})\\
&\quad +B( c_1 c_2 c_3 e^{(\alpha+\beta+\gamma)x} +c_2 c_3 c_4 e^{(\beta+\gamma+\delta)x}
+c_1 c_2 c_4 e^{(\alpha+\beta+\delta)x}+ c_1 c_3 c_4 e^{(\alpha+\gamma+\delta)x})\\
&\quad +C(c_1^2 e^{2\alpha x}+c_2^2 e^{2\beta x}+c_3^2 e^{2\gamma x}+c_4^2 e^{2\delta x})
(c_1 e^{\alpha x}+c_2 e^{\beta x}+c_3 e^{\gamma x}+c_4 e^{\delta x})\\
&\quad +D(c_1 c_2 e^{(\alpha+\beta) x}+c_1 c_3 e^{(\alpha+\gamma) x}+c_1 c_4 e^{(\alpha+\delta)x}
+c_2 c_3 e^{(\beta+\gamma) x}+c_2 c_4 e^{(\beta+\delta) x}+c_3 c_4 e^{(\gamma+\delta)x})\\
&\quad\times(c_1 e^{\alpha x}+c_2 e^{\beta x}+c_3 e^{\gamma x}+c_4 e^{\delta x})\\
&=\frac{A}{563}\sum_{n=0}^\infty T_{n}^{(15,27,48,107)}\frac{(3 x)^n}{n!}
-\frac{B}{563}\sum_{n=0}^\infty \sum_{k=0}^n\binom{n}{k}T_{k}^{(-5,2,13,32)}(-1)^{k}\frac{x^n}{n!}\\
&\quad +\frac{C}{563}\sum_{n=0}^\infty \sum_{k=0}^n\binom{n}{k}2^{n-k}T_{n-k}^{(40,64,215,344)}T_k \frac{x^n}{n!}
+D\sum_{n=0}^\infty \sum_{k=0}^n\binom{n}{k}T_k t_{n-k}\frac{x^n}{n!}\,.
\end{align*}
On the other hand,
$$
\left(\sum_{n=0}^\infty T_n\frac{x^n}{n!}\right)^3=\sum_{k_1+k_2+k_3=n\atop k_1,k_2,k_3\ge 0}\binom{n}{k_1,k_2,k_3}T_{k_1}T_{k_2}T_{k_3}\frac{x^n}{n!}\,.
$$
Comparing the coefficients on both sides, we get the desired result.
\end{proof}

Next, we shall consider the sum of the products of four tetranacci numbers.
We need the following supplementary result.  The proof is similar to that of Lemma \ref{c^3} and omitted.

\begin{Lem}
We have
$$
c_1^4 e^{\alpha x}+c_2^4 e^{\beta x}+c_3^4 e^{\gamma x}+c_4^4 e^{\delta x}=\frac{1}{563^2}\sum_{n=0}^\infty T_n^{(3052,4658,8804,16451)}\frac{x^n}{n!}\,.
$$
\label{c^4}
\end{Lem}
By using Lemmata\ref{c^2}, \ref{cc}, \ref{ccc}, \ref{alg-4}, \ref{c^3}, and \ref{c^4},letting $I=0$ in Lemma \ref{alg-4}, comparing the coefficients on both sides, we can get the following theorem.

\begin{theorem}
\label{tri4}
For $n\ge 0$,
\begin{align*}
&\sum_{k_1+k_2+k_3+k_4=n\atop k_1,k_2,k_3,k_4\ge 0}\binom{n}{k_1,k_2,k_3,k_4}T_{k_1}T_{k_2}T_{k_3}T_{k_4}\\
&=\frac{A}{563^2}4^n T_n^{(3052,4658,8804,16451)}-\frac{B}{563}
+\frac{C}{563}\sum_{k=0}^n\binom{n}{k}3^{n-k}T_{n-k}^{(15,27,48,107)}T_k\\
&\quad+\frac{D}{563^2}\sum_{k=0}^n\binom{n}{k}2^{n}T_{n-k}^{(40,64,215,344)}T_k ^{(40,64,215,344)}\\
&\quad+\frac{E}{563}\sum_{k=0}^n\binom{n}{k}2^{n-k}T_{n-k}^{(40,64,215,344)}t_k +F\sum_{k=0}^n\binom{n}{k}t_{n-k}t_k\\
&\quad+\frac{G}{563}\sum_{k_1+k_2+k_3=n\atop k_1,k_2,k_3\ge0}\binom{n}{k_1,k_2,k_3}T_{k_1}^{(40,64,215,344)}2^{k_1}T_{k_2}T_{k_3}\\
&\quad+H\sum_{k_1+k_2+k_3=n\atop k_1,k_2,k_3\ge 0}\binom{n}{k_1,k_2,k_3}t_{k_1}T_{k_2}T_{k_3}\\
&\quad-\frac{J}{563}\sum_{k_1+k_2+k_3=n\atop k_1,k_2,k_3\ge0}\binom{n}{k_1,k_2,k_3}T_{k_1}^{(-5,2,13,32)}(-1)^{k_1}T_{k_2},
\end{align*}
where $A=-D+E+G+H-3$, $B=-4D+4E+4G+4H-12$,
$C=-E-2G-H+4$,$F=-2D-2G-2H+6$, $J=4D-E+2G-H$,
\begin{align*}
t_n=\frac{1}{2}\left(\sum_{k=0}^n\binom{n}{k}T_k T_{n-k}-\frac{2^n}{563}T_n^{(40,64,215,344)}\right).
\end{align*}
\end{theorem}

\noindent
{\it Remark.}
If $D=E=G=H=0$, then by $A=-3$, $B=-12$, $C=4$, $F=6$ and $J=0$, we have for $n\ge 0$,
\begin{align*}
&\sum_{k_1+k_2+k_3+k_4=n\atop k_1,k_2,k_3,k_4\ge 0}\binom{n}{k_1,k_2,k_3,k_4}T_{k_1}T_{k_2}T_{k_3}T_{k_4}\\
&=-\frac{3}{563^2}4^n T_n^{(3052,4658,8804,16451)}+\frac{12}{563}
+\frac{4}{563}\sum_{k=0}^n\binom{n}{k}3^{n-k}T_{n-k}^{(15,27,48,107)}T_k\\
& \quad +6\sum_{k=0}^n\binom{n}{k}t_{n-k}t_k\,,
\end{align*}

Let
\begin{align*}
&\sum_{n=0}^\infty t_n^{1} \frac{x^n}{n!}\\
&=c_1 c_2 c_3 e^{(\alpha+\beta+\gamma)x}(c_1 e^{\alpha x}+c_2 e^{\beta x}+c_3 e^{\gamma x})
+c_2 c_3 c_4 e^{(\beta+\gamma+\delta)x}(c_2 e^{\beta x}+c_3 e^{\gamma x}+c_4 e^{\delta x})\\
&\quad+c_1 c_2 c_4 e^{(\alpha+\beta+\delta)x}(c_1 e^{\alpha x}+c_2 e^{\beta x}+c_4 e^{\delta x})
+ c_1 c_3 c_4 e^{(\alpha+\gamma+\delta)x}(c_1 e^{\alpha x}+c_3 e^{\gamma x}+c_4 e^{\delta x}).
\end{align*}

By using Lemmata\ref{c^2}, \ref{cc}, \ref{ccc}, \ref{alg-4}, \ref{c^3}, and \ref{c^4}, comparing the coefficients on both sides, we can get the following theorem.

\begin{theorem}
\label{tri4-1}
For $n\ge 0$,$I\neq 0$
 \begin{align*}
&I t_n^{1}
= \sum_{k_1+k_2+k_3+k_4=n\atop k_1,k_2,k_3,k_4\ge 0}\binom{n}{k_1,k_2,k_3,k_4}T_{k_1}T_{k_2}T_{k_3}T_{k_4}\\
&\quad \qquad-\frac{A}{563^2}4^n T_n^{(3052,4658,8804,16451)}+\frac{B}{563}
-\frac{C}{563}\sum_{k=0}^n\binom{n}{k}3^{n-k}T_{n-k}^{(15,27,48,107)}T_k\\
&\quad\qquad-\frac{D}{563^2}\sum_{k=0}^n\binom{n}{k}2^{n}T_{n-k}^{(40,64,215,344)}T_k ^{(40,64,215,344)}\\
&\quad\qquad-\frac{E}{563}\sum_{k=0}^n\binom{n}{k}2^{n-k}T_{n-k}^{(40,64,215,344)}t_k -F\sum_{k=0}^n\binom{n}{k}t_{n-k}t_k\\
&\quad\qquad-\frac{G}{563}\sum_{k_1+k_2+k_3=n\atop k_1,k_2,k_3\ge0}\binom{n}{k_1,k_2,k_3}T_{k_1}^{(40,64,215,344)}2^{k_1}T_{k_2}T_{k_3}\\
&\quad\qquad-H\sum_{k_1+k_2+k_3=n\atop k_1,k_2,k_3\ge 0}\binom{n}{k_1,k_2,k_3}t_{k_1}T_{k_2}T_{k_3}\\
&\quad\qquad+\frac{J}{563}\sum_{k_1+k_2+k_3=n\atop k_1,k_2,k_3\ge0}\binom{n}{k_1,k_2,k_3}T_{k_1}^{(-5,2,13,32)}(-1)^{k_1}T_{k_2},
\end{align*}
where $A=-D+E+G+H-3$,\quad$B=12D+12G-4J-12$,\quad \\
$C=-E-2G-H+4$,\quad
$F=-2D-2G-2H+6$,\quad $I=4D-E+2G-H-J$,
\begin{align*}
t_n=\frac{1}{2}\left(\sum_{k=0}^n\binom{n}{k}T_k T_{n-k}-\frac{2^n}{563}T_n^{(40,64,215,344)}\right).
\end{align*}
\end{theorem}

\noindent
{\it Remark.}
If $D=E=G=H=0$, $J=-1$, then by $A=-3$, $B=-8$, $C=4$, $F=6$ and $I=1$, we have for $n\ge 0$,
\begin{align*}
&t_n^{1}
= \sum_{k_1+k_2+k_3+k_4=n\atop k_1,k_2,k_3,k_4\ge 0}\binom{n}{k_1,k_2,k_3,k_4}T_{k_1}T_{k_2}T_{k_3}T_{k_4}\\
&\quad \quad+\frac{3}{563^2}4^n T_n^{(3052,4658,8804,16451)}-\frac{8}{563}
-\frac{4}{563}\sum_{k=0}^n\binom{n}{k}3^{n-k}T_{n-k}^{(15,27,48,107)}T_k\\
&\quad \quad -6\sum_{k=0}^n\binom{n}{k}t_{n-k}t_k-\frac{1}{563}\sum_{k_1+k_2+k_3=n\atop k_1,k_2,k_3\ge0}\binom{n}{k_1,k_2,k_3}T_{k_1}^{(-5,2,13,32)}(-1)^{k_1}T_{k_2}\,.
\end{align*}
%\begin{proof}
%By Lemmata \ref{c^3} and \ref{c^4} with the proof of Theorem \ref{tri2}, we have
%\begin{align*}
%&(c_1 e^{\alpha x}+c_2 e^{\beta x}+c_3 e^{\gamma x})^4\\
%&=4(c_1^3 e^{3\alpha x}+c_2^3 e^{3\beta x}+c_3^3 e^{3\gamma x})(c_1 e^{\alpha x}+c_2 e^{\beta x}+c_3 e^{\gamma x})\\
%&\quad-3(c_1^4 e^{4\alpha x}+c_2^4 e^{4\beta x}+c_3^4 e^{4\gamma x})\\
%&\quad-6(c_1 c_2 e^{(\alpha+\beta)x}+c_2 c_3 e^{(\beta+\gamma)x}+c_3 c_1 e^{(\gamma+\alpha)x})\\
%&=4\cdot\frac{1}{44}\sum_{n=0}^\infty T_n^{(3,3,5)}\frac{(3 x)^n}{n!}\sum_{n=0}^\infty T_n\frac{x^n}{n!}
%-3\cdot\frac{1}{484}\sum_{n=0}^\infty T_n^{(2,14,21)}\frac{(4 x)^n}{n!}\\
%&\quad+6\left(\sum_{n=0}^\infty\frac{1}{22}\sum_{k=0}^n\binom{n}{k}(-1)^k T_k^{(-1,2,7)}\frac{x^n}{n!}\right)^2\\
%&=\frac{1}{11}\sum_{n=0}^\infty\sum_{k=0}^n\binom{n}{k}3^{n-k}T_{n-k}^{(3,3,5)}T_k\frac{x^n}{n!}
%-\frac{3}{484}\sum_{n=0}^\infty 4^n T_n^{(2,14,21)}\frac{x^n}{n!}\\
%&\quad+\frac{6}{484}\sum_{n=0}^\infty\sum_{k=0}^n\binom{n}{k}\left(\sum_{i=0}^k\binom{k}{i}(-1)^i T_i^{(-1,2,7)}\right)\\
%&\qquad\times\left(\sum_{j=0}^{n-k}\binom{n-k}{j}(-1)^j T_j^{(-1,2,7)}\right)\frac{x^n}{n!}\,.
%\end{align*}
%\end{proof}

%Similarly, since
\section{Convolution identities for five Tetranacci numbers}

 We shall consider the sum of the products of five tetranacci numbers.
We need the following supplementary result.  The proof is similar to that of Lemma \ref{c^3} and omitted.

\begin{Lem}
\begin{align*}
c_1^5 e^{\alpha x}+c_2^5 e^{\beta x}+c_3^5 e^{\gamma x}+c_4^5 e^{\delta x}
=\frac{1}{563^2}\sum_{n=0}^\infty T_n^{(500,1423,2598,4986)}\frac{x^n}{n!}.
\end{align*}
\label{c^5}
\end{Lem}
By using Lemmata\ref{c^2}, \ref{cc}, \ref{ccc}, \ref{alg-5}, \ref{c^3}, \ref{c^4} and  \ref{c^5}, comparing the coefficients on both sides, we can get the following theorems.
\subsection{}
Let $B=C=D=0$, we can obtain the following theorem.
\begin{theorem}
\label{tri5}
For $n\ge 0$,
\begin{align*}
&\sum_{k_1+\cdots+k_5=n\atop k_1,\dots,k_5\ge 0}\binom{n}{k_1,\dots,k_5}T_{k_1}\cdots T_{k_5}\\
&=\frac{A}{563^2}5^n T_n^{(500,1423,2598,4986)}
+\frac{E}{563^2}\sum_{k=0}^n\binom{n}{k}4^{n-k}T_{n-k}^{(3052,4658,8804,16451)}T_k\\
&\quad-\frac{F}{563}\sum_{k=0}^n\binom{n}{k}T_k+G\sum_{k=0}^n\binom{n}{k}t_k^{1} T_{n-k}+\frac{H}{563^2}\sum_{k=0}^n\binom{n}{k}3^{n-k}2^k T_{n-k}^{(15,27,48,107)}T_k^{(40,64,215,344)}\\
\end{align*}
\begin{align*}
&\quad+\frac{I}{563}\sum_{k=0}^n\binom{n}{k}3^{k}T_{k}^{(15,27,48,107)}t_{n-k}\\
&\quad-\frac{J}{563^2}\sum_{k_1+k_2+k_3=n\atop k_1,k_2,k_3\ge0}\binom{n}{k_1,k_2,k_3}(-1)^{k_1}2^{k_2}T_{k_1}^{(-5,2,13,32)}T_{k_2}^{(40,64,215,344)}\\
&\quad-\frac{K}{563}\sum_{k_1+k_2+k_3=n\atop k_1,k_2,k_3\ge0}\binom{n}{k_1,k_2,k_3}(-1)^{k_1}T_{k_1}^{(-5,2,13,32)}t_{k_2}\\
&\quad+\frac{L}{563}\sum_{k_1+k_2+k_3=n\atop k_1,k_2,k_3\ge0}\binom{n}{k_1,k_2,k_3}3^{k_1}T_{k_1}^{(15,27,48,107)}T_{k_2}T_{k_3}\\
&\quad-\frac{M}{563}\sum_{k_1+k_2+k_3+k_4=n\atop k_1,k_2,k_3,k_4\ge0}\binom{n}{k_1,k_2,k_3,k_4}(-1)^{k_1}T_{k_1}^{(-5,2,13,32)}T_{k_2}T_{k_3}\\
&\quad+\frac{N}{563^2}\sum_{k_1+k_2+k_3=n\atop k_1,k_2,k_3\ge0}\binom{n}{k_1,k_2,k_3}2^{k_1}T_{k_1}^{(40,64,215,344)}2^{k_2}T_{k_2}^{(40,64,215,344)}T_{k_3}\\
&\quad+P\sum_{k_1+k_2+k_3=n\atop k_1,k_2,k_3\ge0}\binom{n}{k_1,k_2,k_3}t_{k_1}t_{k_2}T_{k_3}\\
&\quad+\frac{Q}{563}\sum_{k_1+k_2+k_3=n\atop k_1,k_2,k_3\ge 0}\binom{n}{k_1,k_2,k_3}2^{k_1}T_{k_1}^{(40,64,215,344)}t_{k_2}T_{k_3}\\
&\quad+\frac{R}{563}\sum_{k_1+k_2+k_3+k_4=n\atop k_1,k_2,k_3,k_4\ge 0}\binom{n}{k_1,k_2,k_3,k_4}2^{k_1}T_{k_1}^{(40,64,215,344)}T_{k_2}T_{k_3}T_{k_4}\\
&\quad+S\sum_{k_1+k_2+k_3+k_4=n\atop k_1,k_2,k_3,k_4\ge 0}\binom{n}{k_1,k_2,k_3,k_4}t_{k_1}T_{k_2}T_{k_3}T_{k_4},
\end{align*}
where  \begin{align*}
&A=I+2L+2N+P+2Q+6R+4S-14,\\
&E=-I-2L-N-Q-3R-S+5, \\
&F=4G+I+2L+6N+5P+6Q+18R+16S-50,\\
&H=-L-2N-P-Q-4R-3S+10,\\
&J=-G-I-2L-M-2P-3Q-6R-7S+20, \\
&K=-2G-2M-2N-5P-2Q-6R-12S+30,
\end{align*}
$t_n$ and $t_n^{1}$ are same as those in  theorem \ref{cc} and theorem \ref{tri4-1}, respectively.
\end{theorem}

\noindent
{\it Remark.}
If $G=I=L=M=N=P=Q=R=S=0$, then by $A=-14$, $E=5$, $F=-50$, $H=10$, $J=20$ and $K=30$, we have for $n\ge 0$,
\begin{align*}
&\sum_{k_1+\cdots+k_5=n\atop k_1,\dots,k_5\ge 0}\binom{n}{k_1,\dots,k_5}T_{k_1}\cdots T_{k_5}\\
&=-\frac{14}{563^2}5^n T_n^{(500,1423,2598,4986)}
+\frac{5}{563^2}\sum_{k=0}^n\binom{n}{k}4^{n-k}T_{n-k}^{(3052,4658,8804,16451)}T_k\\
&\quad+\frac{50}{563}\sum_{k=0}^n\binom{n}{k}T_k+\frac{10}{563^2}\sum_{k=0}^n\binom{n}{k}3^{n-k}2^k T_{n-k}^{(15,27,48,107)}T_k^{(40,64,215,344)}\\
&\quad-\frac{20}{563^2}\sum_{k_1+k_2+k_3=n\atop k_1,k_2,k_3\ge0}\binom{n}{k_1,k_2,k_3}(-1)^{k_1}2^{k_2}T_{k_1}^{(-5,2,13,32)}T_{k_2}^{(40,64,215,344)}\\
&\quad-\frac{30}{563}\sum_{k_1+k_2+k_3=n\atop k_1,k_2,k_3\ge0}\binom{n}{k_1,k_2,k_3}(-1)^{k_1}T_{k_1}^{(-5,2,13,32)}t_{k_2}\,.
\end{align*}

\subsection{}
Let $B\neq 0, $ $C=D=0$, we can obtain the following theorem.\\
Let  \begin{align*}
&\sum_{n=0}^\infty t_n^{2} \frac{x^n}{n!}\\
&=c_1 c_2 c_3 e^{(\alpha+\beta+\gamma)x}(c_1c_2 e^{(\alpha+\beta) x}+c_2c_3 e^{(\beta+\gamma) x}
+c_3 c_1e^{(\gamma+\alpha) x})+\cdots\\
&\quad +c_2 c_3 c_4 e^{(\beta+\gamma+\delta)x}(c_2c_3 e^{(\beta+\gamma) x}+c_3 c_4e^{(\gamma+\delta) x}
+c_4 c_2e^{(\delta+\beta) x}).
\end{align*}

\begin{theorem}
\label{tri5-2}
For $n\ge 0$,
 \begin{align*}
&B t_n^{2}
=\sum_{k_1+\cdots+k_5=n\atop k_1,\dots,k_5\ge 0}\binom{n}{k_1,\dots,k_5}T_{k_1}\cdots T_{k_5}
-\frac{A}{563^2}5^n T_n^{(500,1423,2598,4986)}-\cdots\\
&\qquad \quad -S\sum_{k_1+k_2+k_3+k_4=n\atop k_1,k_2,k_3,k_4\ge 0}\binom{n}{k_1,k_2,k_3,k_4}t_{k_1}T_{k_2}T_{k_3}T_{k_4},
 \end{align*}
where  \begin{align*}
&A=-G-J-M+2N-P-Q-3S+6,\\
&B=-2G-K-2M-2N-5P-2Q-6R-12S+30,\\
&E=G+J+M-N+2P+2Q+3R+6S-15,\\
&F=-3G-J-3K-7M-12P-3Q-6R-27S+60, \\
&H=-L-2N-P-Q-4R-3S+10, \\
&I=-G-J-2L-M-2P-3Q-6R-7S+20\,,
  \end{align*}
$t_n$ and $t_n^{1}$ are same as those in  theorem \ref{cc} and theorem \ref{tri4-1}, respectively.
\end{theorem}

\noindent
{\it Remark.}
If $G=J=K=L=M=N=P=Q=R=S=0$, then by $A=6$, $B=30$, $E=-15$, $F=60$, $H=10$ and $I=20$, we have for $n\ge 0$,
\begin{align*}
&30 t_n^{2}
=\sum_{k_1+\cdots+k_5=n\atop k_1,\dots,k_5\ge 0}\binom{n}{k_1,\dots,k_5}T_{k_1}\cdots T_{k_5}
-\frac{6}{563^2}5^n T_n^{(500,1423,2598,4986)}\\
&\qquad\quad +\frac{15}{563^2}\sum_{k=0}^n\binom{n}{k}4^{n-k}T_{n-k}^{(3052,4658,8804,16451)}T_k+
\frac{60}{563}\sum_{k=0}^n\binom{n}{k}T_k\\
&\qquad\quad-\frac{10}{563^2}\sum_{k=0}^n\binom{n}{k}3^{n-k}2^k T_{n-k}^{(15,27,48,107)}T_k^{(40,64,215,344)}
-\frac{20}{563}\sum_{k=0}^n\binom{n}{k}3^{k}T_{k}^{(15,27,48,107)}t_{n-k}\,.
 \end{align*}

\subsection{}
Let $C\neq 0, $ $B=D=0,$ we can obtain the following theorem.\\
Let  \begin{align*}
&\sum_{n=0}^\infty t_n^{3} \frac{x^n}{n!}\\
&=c_1 c_2 c_3 e^{(\alpha+\beta+\gamma)x}(c_1^2 e^{2\alpha x}+c_2^2 e^{2 \beta x}
+c_3^2e^{2\gamma x})+\cdots\\
&\quad +c_2 c_3 c_4 e^{(\beta+\gamma+\delta)x}(c_2^2 e^{2\beta x}+c_3^2e^{2\gamma x}
+c_4^2e^{2\delta x}).
\end{align*}

\begin{theorem}
For $n\ge 0$,
 \begin{align*}
&C t_n^{3}
=\sum_{k_1+\cdots+k_5=n\atop k_1,\dots,k_5\ge 0}\binom{n}{k_1,\dots,k_5}T_{k_1}\cdots T_{k_5}
-\frac{A}{563^2}5^n T_n^{(500,1423,2598,4986)}-\cdots\\
&\qquad \quad -S\sum_{k_1+k_2+k_3+k_4=n\atop k_1,k_2,k_3,k_4\ge 0}\binom{n}{k_1,k_2,k_3,k_4}t_{k_1}T_{k_2}T_{k_3}T_{k_4},
 \end{align*}
\label{tri5-3}
where  \begin{align*}
&A=I+2L+2N+P+2Q+6R+4S-14,\\
&C=-G-I-J-2L-M-2P-3Q-6R-7S+20,\\
&E=-I-2L-N-Q-3R-S+5,\\
&F=3G-J-M+6N+3P+3Q+12R+9S-30,\\
&H=-L-2N-P-Q-4R-3S+10,\\
&K=-2G-2M-2N-5P-2Q-6R-12S+30,
  \end{align*}
$t_n$ and $t_n^{1}$ are same as those in  theorem \ref{cc} and theorem \ref{tri4-1}, respectively.
\end{theorem}

\noindent
{\it Remark.}
If $G=I=J=L=M=N=P=Q=R=S=0$, then by $A=-14$, $C=20$, $E=5$, $F=-30$, $H=10$ and $K=30$, we have for $n\ge 0$,
\begin{align*}
&20 t_n^{3}
=\sum_{k_1+\cdots+k_5=n\atop k_1,\dots,k_5\ge 0}\binom{n}{k_1,\dots,k_5}T_{k_1}\cdots T_{k_5}
+\frac{14}{563^2}5^n T_n^{(500,1423,2598,4986)}\\
&\qquad\quad-\frac{5}{563^2}\sum_{k=0}^n\binom{n}{k}4^{n-k}T_{n-k}^{(3052,4658,8804,16451)}T_k+
\frac{30}{563}\sum_{k=0}^n\binom{n}{k}T_k\\
&\qquad\quad-\frac{10}{563^2}\sum_{k=0}^n\binom{n}{k}3^{n-k}2^k T_{n-k}^{(15,27,48,107)}T_k^{(40,64,215,344)}\\
&\qquad\quad+\frac{30}{563}\sum_{k_1+k_2+k_3=n\atop k_1,k_2,k_3\ge0}\binom{n}{k_1,k_2,k_3}(-1)^{k_1}T_{k_1}^{(-5,2,13,32)}t_{k_2}\,.
 \end{align*}

\subsection{}
Let $D\neq 0, $ $B=C=0,$ we can obtain the following theorem.\\
Let  \begin{align*}
&\sum_{n=0}^\infty t_n^{4} \frac{x^n}{n!}\\
&=c_1 c_2 c_3 e^{(\alpha+\beta+\gamma)x}(c_1 e^{\alpha x}+c_2 e^{\beta x}
+c_3e^{\gamma x})^2+\cdots\\
&\quad +c_2 c_3 c_4 e^{(\beta+\gamma+\delta)x}(c_2 e^{\beta x}+c_3e^{\gamma x}
+c_4e^{\delta x})^2.
\end{align*}

\begin{theorem}
For $n\ge 0$,
 \begin{align*}
&D t_n^{4}
=\sum_{k_1+\cdots+k_5=n\atop k_1,\dots,k_5\ge 0}\binom{n}{k_1,\dots,k_5}T_{k_1}\cdots T_{k_5}
-\frac{A}{563^2}5^n T_n^{(500,1423,2598,4986)}-\cdots\\
&\qquad\quad -S\sum_{k_1+k_2+k_3+k_4=n\atop k_1,k_2,k_3,k_4\ge 0}\binom{n}{k_1,k_2,k_3,k_4}t_{k_1}T_{k_2}T_{k_3}T_{k_4},
 \end{align*}
\label{tri5-4}
where  \begin{align*}
&A=I+2L+2N+P+2Q+6R+4S-14,\\
&D=-G-I-J-2L-M-2P-3Q-6R-7S+20,\\
&E=-I-2L-N-Q-3R-S+5,\\
&F=-3G-6I-7J-7M+6N-12L-9P-15Q-24R-33S+90,\\
&H=-L-2N-P-Q-4R-3S+10,\\
&K=2I+2J-2N+4L-P+4Q+6R+2S-10,
  \end{align*}
$t_n$ and $t_n^{1}$ are same as those in  theorem \ref{cc} and theorem \ref{tri4-1}, respectively.
\end{theorem}

\noindent
{\it Remark.}
If $G=I=J=L=M=N=P=Q=R=S=0$, then by $A=-14$, $D=20$, $E=5$, $F=90$, $H=10$ and $K=-10$, we have for $n\ge 0$,
\begin{align*}
&20 t_n^{4}
=\sum_{k_1+\cdots+k_5=n\atop k_1,\dots,k_5\ge 0}\binom{n}{k_1,\dots,k_5}T_{k_1}\cdots T_{k_5}
+\frac{14}{563^2}5^n T_n^{(500,1423,2598,4986)}\\
&\qquad\quad-\frac{5}{563^2}\sum_{k=0}^n\binom{n}{k}4^{n-k}T_{n-k}^{(3052,4658,8804,16451)}T_k
+\frac{90}{563}\sum_{k=0}^n\binom{n}{k}T_k\\
&\qquad\quad-\frac{10}{563^2}\sum_{k=0}^n\binom{n}{k}3^{n-k}2^k T_{n-k}^{(15,27,48,107)}T_k^{(40,64,215,344)}\\
&\qquad\quad-\frac{10}{563}\sum_{k_1+k_2+k_3=n\atop k_1,k_2,k_3\ge0}\binom{n}{k_1,k_2,k_3}(-1)^{k_1}T_{k_1}^{(-5,2,13,32)}t_{k_2}\,.
 \end{align*}

\section{More general results}

We shall consider the general case of Lemmata \ref{c^2}, \ref{c^3} and \ref{c^4}.
Similarly to the proof of Lemma \ref{c^2}, for tetranacci-type numbers $s_{1,n}^{(n)}$, satisfying the recurrence relation $s_{1,k}^{(n)}=s_{1,k-1}^{(n)}+s_{1,k-2}^{(n)}+s_{1,k-3}^{(n)}+s_{1,k-4}^{(n)}$ ($k\ge 4$) with given initial values $s_{1,0}^{(n)}$, $s_{1,1}^{(n)}$, $s_{1,2}^{(n)}$ and $s_{1,3}^{(n)}$, we have the form
$$
d_1^{(n)}e^{\alpha x}+d_2^{(n)}e^{\beta x}+d_3^{(n)}e^{\gamma x}+d_4^{(n)}e^{\delta x}=\sum_{k=0}^\infty s_{1,k}^{(n)}\frac{x^k}{k!}\,.
$$

%Next we shall give some notation.
%$d_1=Ac_1^2$ is denoted by $d_1^{n}=A_1^{n}c_1^n$. Notice that, only for $c_1$, n is a power.
\begin{theorem}
\label{c^n}
For $n\ge 1$, we have
$$
c_1^n e^{\alpha x}+c_2^n e^{\beta x}+c_3^n e^{\gamma x}+c_4^n e^{\delta x}
=\frac{1}{A_1^{(n)}}\sum_{k=0}^\infty T_{1,k}^{(s_{1,0}^{(n)},s_{1,1}^{(n)},s_{1,2}^{(n)},s_{1,3}^{(n)})}\frac{x^k}{k!}\,,
$$
where $s_{1,0}^{(n)}$,~ $s_{1,1}^{(n)}$,~ $s_{1,2}^{(n)}$,~ $s_{1,3}^{(n)}$ and $A_1^{(n)}$ satisfy the recurrence relations:
\begin{align*}
& s_{1,0}^{(n)}=\pm{\rm lcm}(b_1,b_2,b_3),\quad s_{1,1}^{(n)}=Ms_{1,0}^{(n)},\quad  s_{1,2}^{(n)}=Ns_{1,0}^{(n)}, \quad s_{1,3}^{(n)}=Ps_{1,0}^{(n)},\\
& A_1^{(n)}=\frac{A_1^{(n-1)}}{s_{1,2}^{(n-1)}}(4s_{1,3}^{(n)}-3s_{1,2}^{(n)}-2s_{1,1}^{(n)}-s_{1,0}^{(n)}),
 \end{align*}
 $b_1$,~$b_2$,~ $b_3$,~ $M$,~ $N$ and  $P$ are determined in the proof.
\end{theorem}

\begin{proof}
By $d_1^{(n)}=A_1^{(n)}c_1^n$,~$d_1^{(n-1)}=A_1^{(n-1)}c_1^{n-1}$,
\begin{align*}
&d_1^{(n)}=\frac{s_{1,0}^{(n)}\beta\gamma\delta+s_{1,2}^{(n)}(\beta+\gamma+\delta)
-s_{1,3}^{(n)}-s_{1,1}^{(n)}(\beta\gamma+\beta\delta+\gamma\delta)}
{(\beta-\alpha)(\gamma-\alpha)(\delta-\alpha)},\\
&c_1=\frac{2-(\beta+\gamma+\delta)+(\beta\gamma+\gamma\delta+\delta\beta)}
{(\alpha-\beta)(\alpha-\gamma)(\alpha-\delta)}\\
&\quad=\frac{\alpha^2}{(\alpha-\beta)(\alpha-\gamma)(\alpha-\delta)}\\
&\quad=\frac{1}{-\alpha^3+6\alpha-1},
\end{align*}
 we can obtain the following recurrence relation:
\begin{align*}
&A=-3s_{1,1}^{(n-1)}-s_{1,2}^{(n-1)}+s_{1,3}^{(n-1)},\quad B=5s_{1,1}^{(n-1)}+5s_{1,2}^{(n-1)}-5s_{1,3}^{(n-1)},\\
&C=5s_{1,1}^{(n-1)}-2s_{1,2}^{(n-1)}+2s_{1,3}^{(n-1)}, \quad D=5s_{1,1}^{(n-1)}-s_{1,2}^{(n-1)}+s_{1,3}^{(n-1)},\\
&E=s_{1,0}^{(n-1)}-s_{1,1}^{(n-1)}, \quad F=-5s_{1,0}^{(n-1)},\quad
 G=2s_{1,0}^{(n-1)}+6s_{1,1}^{(n-1)},
\end{align*}
\begin{align*}
&H=s_{1,0}^{(n-1)}-s_{1,1}^{(n-1)}, \quad I=4s_{1,1}^{(n-1)}+6s_{1,2}^{(n-1)}, \quad J=-3s_{1,1}^{(n-1)}-5s_{1,2}^{(n-1)},\\
&K=-2s_{1,1}^{(n-1)}+2s_{1,2}^{(n-1)},\quad L=-s_{1,1}^{(n-1)}+s_{1,2}^{(n-1)},\\
&M=\frac{(LA-DI)(FA-BE)-(HA-DE)(JA-BI)}{(GA-CE)(JA-BI)-(KA-CI)(FA-BE)},\\
&N=\frac{M(KA-CI)+(LA-DI)}{BI-JA}, \quad  P=-\frac{1}{A}(BN+CM+D),\\
&M=\frac{a_1}{b_1}, \quad  N=\frac{a_2}{b_2}, \quad  P=\frac{a_3}{b_3},\quad \hbox{with}\quad \gcd(a_i,b_i)=1,\\
&s_{1,0}^{(n)}=\pm{\rm lcm}(b_1,b_2,b_3), \quad s_{1,1}^{(n)}=Ms_{1,0}^{(n)},\quad s_{1,2}^{(n)}=Ns_{1,0}^{(n)}, \quad s_{1,3}^{(n)}=Ps_{1,0}^{(n)},\\
&A_1^{(n)}=\frac{A_1^{(n-1)}}{s_{1,2}^{(n-1)}}(4s_{1,3}^{(n)}-3s_{1,2}^{(n)}-2s_{1,1}^{(n)}-s_{1,0}^{(n)}).
\end{align*}
 We choose the symbol of $s_{1,0}^{(n)}$ such that for some $k_0$,$\forall k \geq k_0$,
$T_{1,k}^{(s_{1,0}^{(n)},s_{1,1}^{(n)},s_{1,2}^{(n)},s_{1,3}^{(n)})}$ is positive.
\end{proof}

\bigskip

Next we shall consider the general case of Lemma \ref{ccc}.
Similarly to the proof of Lemma \ref{ccc}, for tetranacci-type numbers $s_{1,k}^{(n)}$, satisfying the recurrence relation $s_{1,k}^{(n)}=s_{1,k-1}^{(n)}+s_{1,k-2}^{(n)}+s_{1,k-3}^{(n)}+s_{1,k-4}^{(n)}$ ($k\ge 4$) with given initial values $s_{1,0}^{(n)}$, $s_{1,1}^{(n)}$, $s_{1,2}^{(n)}$ and $s_{1,3}^{(n)}$, we have the form
$$
r_1^{(n)}e^{\alpha x}+r_2^{(n)}e^{\beta x}+r_3^{(n)}e^{\gamma x}+r_4^{(n)}e^{\delta x}=\sum_{k=0}^\infty s_{1,k}^{(n)}\frac{x^k}{k!}\,,
$$
where  $r_1^{(n)}$,~ $r_2^{(n)}$,~ $r_3^{(n)}$ and $r_4^{(n)}$ are determined by solving the system of the equations.
%Similarly,by $d_1^{n}=A_2^{n}c_2^nc_3^nc_4^n$, we can obtain the following theorem.
 %The proof is indeed a great mount of calculation,for this case, we only list the theorem without a proof.
\begin{theorem}
\label{ccc^n}
\begin{align*}
c_2^nc_3^nc_4^n e^{\alpha x}+c_1^nc_3^nc_4^ne^{\beta x}
+c_1^nc_2^nc_4^n e^{\gamma x}+c_1^nc_2^nc_3^n e^{\delta x}
=\frac{1}{A_2^{(n)}}\sum_{k=0}^\infty T_{2,k}^{(s_{2,0}^{(n)},s_{2,1}^{(n)},s_{2,2}^{(n)},s_{2,3}^{(n)})}\frac{x^k}{k!}\,,
\end{align*}
where $s_{2,0}^{(n)}$, ~$s_{2,1}^{(n)}$, ~$s_{2,2}^{(n)}$,~ $s_{2,3}^{(n)}$ and $A_2^{(n)}$ satisfy the recurrence relations:
\begin{align*}
&s_{2,0}^{(n)}=\pm{\rm lcm}(b_1,b_2,b_3), \quad s_{2,1}^{(n)}=Ms_{2,0}^{(n)}, \quad s_{2,2}^{(n)}=Ns_{2,0}^{(n)},  \quad s_{2,3}^{(n)}=Ps_{2,0}^{(n)}, \\
& A_2^{(n)}=\frac{A_2^{(n-1)}}{2s_{2,2}^{(n-1)}-s_{2,3}^{(n-1)}}(-16s_{2,3}^{(n)}+103s_{2,2}^{(n)}
-157s_{2,1}^{(n)}-10s_{2,0}^{(n)})\,,
\end{align*}
 $b_1$,$b_2$, $b_3$, $M$, $N$ and  $P$ are determined in the proof.
\end{theorem}
\begin{proof}
By $r_1^{(n)}=A_2^{(n)}c_2^nc_3^nc_4^n$, we can obtain the following recurrence relation:
\begin{align*}
&A=-16s_{2,0}^{(n-1)}-16s_{2,1}^{(n-1)}+158s_{2,2}^{(n-1)}-71s_{2,3}^{(n-1)},\\
&B=103s_{2,0}^{(n-1)}+103s_{2,1}^{(n-1)}-243s_{2,2}^{(n-1)}+70s_{2,3}^{(n-1)},\\
&C=-157s_{2,0}^{(n-1)}-157s_{2,1}^{(n-1)}-209s_{2,2}^{(n-1)}+183s_{2,3}^{(n-1)},\\
&D=-10s_{2,0}^{(n-1)}-10s_{2,1}^{(n-1)}-42s_{2,2}^{(n-1)}+26s_{2,3}^{(n-1)},\\
&E=32s_{2,0}^{(n-1)}+16s_{2,1}^{(n-1)}-330s_{2,2}^{(n-1)}+157s_{2,3}^{(n-1)},\\
&F=-206s_{2,0}^{(n-1)}-103s_{2,1}^{(n-1)}+365s_{2,2}^{(n-1)}-131s_{2,3}^{(n-1)},\\
&G=314s_{2,0}^{(n-1)}+157s_{2,1}^{(n-1)}+351s_{2,2}^{(n-1)}-254s_{2,3}^{(n-1)},\\
&H=20s_{2,0}^{(n-1)}+10s_{2,1}^{(n-1)}+216s_{2,2}^{(n-1)}-113s_{2,3}^{(n-1)},\\
&I=32s_{2,1}^{(n-1)}-36s_{2,2}^{(n-1)}+10s_{2,3}^{(n-1)},  \quad
J=-206s_{2,1}^{(n-1)}+91s_{2,2}^{(n-1)}+6s_{2,3}^{(n-1)},\\
&K=314s_{2,1}^{(n-1)}+69s_{2,2}^{(n-1)}-113s_{2,3}^{(n-1)},  \quad
L=20s_{2,1}^{(n-1)}-304s_{2,2}^{(n-1)}+147s_{2,3}^{(n-1)}\,,
 \end{align*}
 \begin{align*}
&M=\frac{(LA-DI)(FA-BE)-(HA-DE)(JA-BI)}{(GA-CE)(JA-BI)-(KA-CI)(FA-BE)},\\
&N=\frac{M(GA-CE)+(HA-DE)}{BE-FA}, \quad P=-\frac{1}{A}(BN+CM+D),\\
&M=\frac{a_1}{b_1},\quad N=\frac{a_2}{b_2}, \quad P=\frac{a_3}{b_3}, \quad  \gcd(a_i,b_i)=1,\\
&s_{2,0}^{(n)}=\pm{\rm lcm}(b_1,b_2,b_3),\quad s_{2,1}^{(n)}=Ms_{2,0}^{(n)},\quad s_{2,2}^{(n)}=Ns_{2,0}^{(n)}, \quad s_{2,3}^{(n)}=Ps_{2,0}^{(n)},\\
&A_2^{(n)}=\frac{A_2^{(n-1)}}{2s_{2,2}^{(n-1)}-s_{2,3}^{(n-1)}}(-16s_{2,3}^{(n)}+103s_{2,2}^{(n)}-157s_{2,1}^{(n)
}-10s_{2,0}^{(n)}).
\end{align*}
We choose the symbol of $s_{2,0}^{(n)}$ such that for some $k_0$,$\forall k \geq k_0$,
$T_{2,k}^{(s_{2,0}^{(n)},s_{2,1}^{(n)},s_{2,2}^{(n)},s_{2,3}^{(n)})}$ is positive.
\end{proof}

As application, we compute some values of $s_{2,0}^{(n)},~s_{2,1}^{(n)},~s_{2,2}^{(n)},~A_2^{(n)}$ for some $n$.
For $n=2$,  we have
\begin{align*}
&A=-170,\quad B=-1228,\quad C=3610,\quad D=316,\quad E=606, \quad F=1377,\\
& G=-4821,\quad  H=-888,\quad
I=-84,\quad  J=963, \quad K=-2091, \quad L=792,\\
&M=-\frac{34}{15}, \quad N=-6, \quad P=-\frac{44}{15},\\
&s_{2,0}^{2}=-15, \quad s_{2,1}^{2}=34, \quad s_{2,2}^{2}=90, \quad s_{2,3}^{2}=44, \quad A_2^{2}=563^2,\\
&c_2^2 c_3^2 c_4^2 e^{\alpha x}+c_1^2 c_3^2 c_4^2 e^{\beta x}
+c_1^2 c_2^2 c_4^2 e^{\gamma x}+c_1^2 c_2^2 c_3^2 e^{\delta x}
=\frac{1}{563^2}\sum_{k=0}^\infty T_{2,k}^{(-15,34,90,44)}\frac{x^k}{k!}.
\end{align*}

For $n=3$,  we have
\begin{align*}
&A=10792,\quad  B=-16833,\quad  C=13741,\quad  D=-2826, \quad E=-22728, \quad F=26674,\\
& G=21042,\quad H=14508, \quad I=-1712, \quad J=1450, \quad K=11914, \quad L=-20212,\\
&M=\frac{353}{175}, \quad N=-\frac{21}{25}, \quad P=\frac{38}{25},\\
&s_{2,0}^{3}=175, \quad s_{2,1}^{3}=353, \quad s_{2,2}^{3}=-147,\quad s_{2,3}^{3}=266,\quad A_2^{3}=-563^3,\\
&c_2^3 c_3^3 c_4^3 e^{\alpha x}+c_1^3 c_3^3 c_4^3 e^{\beta x}
+c_1^3 c_2^3 c_4^3 e^{\gamma x}+c_1^3 c_2^3 c_3^3 e^{\delta x}
=-\frac{1}{563^3}\sum_{k=0}^\infty T_{2,k}^{(175,353,-147,266)}\frac{x^k}{k!}.
\end{align*}

 We can obtain more convolution identities for any fixed $n$, but we only some of the results.The proof of next eight theorems are similar to the proofs of theorem (lemma) \ref{cc}, \ref{tri3}, \ref{tri4}, \ref{tri4-1}, \ref{tri5}, \ref{tri5-2}, \ref{tri5-3} and \ref{tri5-4}, and omitted.

Let $$c_1^n c_2^ne^{(\alpha+\beta) x}+\cdots +c_3^n c_4^n e^{(\gamma+\delta)x}
=\sum_{k=0}^\infty t_{1,k}^{(n)}  \frac{x^k}{k!},$$\\
 then by previous algebraic identities ,we can obtain the following theorems.

\begin{theorem}
For $m\ge 0$, $n\ge 1$,
$$
c_1^n c_2^ne^{(\alpha+\beta) x}+\cdots +c_3^n c_4^n e^{(\gamma+\delta)x}
=\sum_{k=0}^\infty t_{1,m}^{(n)} \frac{x^m}{m!},
$$
\label{cc^n}
where
$$
t_{1,m}^{(n)} =\frac{1}{2}\left(\frac{1}{(A_1^{(n)})^2}\sum_{k=0}^m\binom{m}{k}
T_{1,k}^{(s_{1,0}^{n},s_{1,1}^{(n)},s_{1,2}^{(n)},s_{1,3}^{(n)})} T_{1,m-k}^{(s_{1,0}^{(n)},s_{1,1}^{(n)},s_{1,2}^{(n)},s_{1,3}^{(n)})}
-\frac{2^m}{A_1^{(2n)}}T_{1,m}^{(s_{1,0}^{(n)},s_{1,1}^{(n)},s_{1,2}^{(n)},s_{1,3}^{(n)})}\right).
$$
\end{theorem}

\begin{theorem}
\label{tri3-n}
For $m\ge 0$, $n\ge 1$,
\begin{align*}
&\frac{1}{(A_1^{(n)})^3}\sum_{k_1+k_2+k_3=m\atop k_1,k_2,k_3\ge 0}\binom{m}{k_1,k_2,k_3}T_{1,k_1}^{(s_{1,0}^{(n)},s_{1,1}^{(n)},s_{1,2}^{(n)},s_{1,3}^{(n)})}
T_{1,k_2}^{(s_{1,0}^{(n)},s_{1,1}^{(n)},s_{1,2}^{(n)},s_{1,3}^{(n)})}
T_{1,k_3}^{(s_{1,0}^{(n)},s_{1,1}^{(n)},s_{1,2}^{(n)},s_{1,3}^{(n)})}\\
&=\frac{A}{A_1^{(3n)}}3^mT_{1,m}^{(s_{1,0}^{(3n)},s_{1,1}^{(3n)},s_{1,2}^{(3n)},s_{1,3}^{(3n)})}
+\frac{B}{A_2^{(n)}}\sum_{k=0}^m\binom{m}{k}T_{2,k}^{(s_{2,0}^{(n)},s_{2,1}^{(n)},s_{2,2}^{(n)},s_{2,3}^{(n)})}(-1)^{k}\\
&\quad+\frac{C}{{A_1^{(n)}}A_1^{(2n)}}\sum_{k=0}^m\binom{m}{k}T_{1,k}^{(s_{1,0}^{(n)},s_{1,1}^{(n)},s_{1,2}^{(n)},
s_{1,3}^{(n)})}
2^{m-k}T_{1,m-k}^{(s_{1,0}^{(2n)},s_{1,1}^{(2n)},s_{1,2}^{(2n)},s_{1,3}^{(2n)})}\\
&\quad+\frac{D}{A_1^{(n)}}\sum_{k=0}^m\binom{m}{k}T_{1,k}^{(s_{1,0}^{(n)},s_{1,1}^{(n)},s_{1,2}^{(n)},s_{1,3}^{(n)})} t_{1,m-k}^{(n)}\,,
\end{align*}
where $$A=D-2,\quad  B=-3D+6,\quad C=-D+3,$$
$t_{1,m}^{(n)}$ is determined in  theorem \ref{cc^n}.
\end{theorem}

\begin{theorem}
\label{tri4-n}
For $m\ge 0$, $n\ge 1$,
\begin{align*}
&\frac{1}{(A_1^{(n)})^4}\sum_{k_1+k_2+k_3+k_4=m\atop k_1,k_2,k_3,k_4\ge 0}\binom{m}{k_1,k_2,k_3,k_4}
 T_{1,k_1}^{(s_{1,0}^{(n)},s_{1,1}^{(n)},s_{1,2}^{(n)},s_{1,3}^{(n)})}\cdots
 T_{1,k_4}^{(s_{1,0}^{(n)},s_{1,1}^{(n)},s_{1,2}^{(n)},s_{1,3}^{(n)})}  \\
&=\frac{A}{A_1^{(4n)}}4^m T_{1,m}^{(s_{1,0}^{(4n)},\cdots,s_{1,3}^{(4n)})}+B(\frac{-1}{563})^n
+\frac{C}{A_1^{(3n)}A_1^{(n)}}\sum_{k=0}^m\binom{m}{k}3^{m-k}T_{1,m-k}^{(s_{1,0}^{(3n)},\cdots,s_{1,3}^{(3n)})}
T_{1,k}^{(s_{1,0}^{(n)},\cdots,s_{1,3}^{(n)})}\\
&\quad+\frac{D}{(A_1^{(2n)})^2}\sum_{k=0}^m\binom{m}{k}2^{m}T_{1,k}^{(s_{1,0}^{(2n)},\cdots,s_{1,3}^{(2n)})}
T_{1,m-k}^{(s_{1,0}^{(2n)},\cdots,s_{1,3}^{(2n)})}\\
&\quad+\frac{E}{A_1^{(2n)}}\sum_{k=0}^m\binom{m}{k}2^{m-k}T_{1,m-k}^{(s_{1,0}^{(2n)},\cdots,s_{1,3}^{(2n)})}
t_{1,m-k}^{(n)} +F\sum_{k=0}^m\binom{m}{k}t_{1,k}^{(n)}t_{1,m-k}^{(n)}\\
&\quad+\frac{G}{A_1^{(2n)}(A_1^{(n)})^2}\sum_{k_1+k_2+k_3=m\atop k_1,k_2,k_3\ge0}\binom{m}{k_1,k_2,k_3}T_{1,k_1}^{(s_{1,0}^{(2n)},\cdots,s_{1,3}^{(2n)})}
2^{k_1}T_{1,k_2}^{(s_{1,0}^{(n)},\cdots,s_{1,3}^{(n)})}T_{1,k_3}^{(s_{1,0}^{(n)},\cdots,s_{1,3}^{(n)})}\\
&\quad+\frac{H}{(A_1^{(n)})^2}\sum_{k_1+k_2+k_3=m\atop k_1,k_2,k_3\ge 0}\binom{m}{k_1,k_2,k_3}t_{1,k_1}^{(n)}
T_{1,k_2}^{(s_{1,0}^{(n)},\cdots,s_{1,3}^{(n)})}T_{1,k_3}^{(s_{1,0}^{(n)},\cdots,s_{1,3}^{(n)})}\\
&\quad+\frac{J}{A_2^{(n)}A_1^{(n)}}\sum_{k_1+k_2+k_3=m\atop k_1,k_2,k_3\ge0}\binom{m}{k_1,k_2,k_3}
(-1)^{k_1}T_{2,k_1}^{(s_{2,0}^{(n)},\cdots,s_{2,3}^{(n)})}T_{1,k_2}^{(s_{1,0}^{(n)},\cdots,s_{1,3}^{(n)})},
\end{align*}
where $A=-D+E+G+H-3$,\quad $B=-4D+4E+4G+4H-12$,\\
$C=-E-2G-H+4$,\quad $F=-2D-2G-2H+6$,\quad $J=4D-E+2G-H$,
$t_{1,m}^{(n)}$ is determined in  theorem \ref{cc^n}.
\end{theorem}

Let
\begin{align*}
&\sum_{k=0}^\infty t_{2,k}^{(n)} \frac{x^k}{k!}
=c_1^n c_2^n c_3^n e^{(\alpha+\beta+\gamma)x}(c_1^n e^{\alpha x}+c_2^n e^{\beta x}+c_3^n e^{\gamma x})
+\cdots\\
&\quad+ c_2^n c_3^n c_4^n e^{(\alpha+\gamma+\delta)x}(c_2^n e^{\alpha x}+c_3^n e^{\gamma x}+c_4^n e^{\delta x}).
\end{align*}

\begin{theorem}
\label{tri7}
For $m\ge 0$, $n\ge 1$,$I\neq 0$,
\begin{align*}
&I t_{2,m}^{(n)}=\frac{1}{(A_1^{(n)})^4}\sum_{k_1+k_2+k_3+k_4=m\atop k_1,k_2,k_3,k_4\ge 0}\binom{m}{k_1,k_2,k_3,k_4}
 T_{1,k_1}^{(s_{1,0}^{(n)},\cdots,s_{1,3}^{(n)})} \cdots T_{1,k_4}^{(s_{1,0}^{(n)},\cdots,s_{1,3}^{(n)})} \\
 &\quad-\frac{A}{A_1^{(4n)}}4^m T_{1,m}^{(s_{1,0}^{(4n)},\cdots,s_{1,3}^{(4n)})}- \cdots \\
 &\quad-\frac{J}{A_2^{(n)}A_1^{(n)}}\sum_{k_1+k_2+k_3=m\atop k_1,k_2,k_3\ge0}\binom{m}{k_1,k_2,k_3}
(-1)^{k_1}T_{2,k_1}^{(s_{2,0}^{(n)},\cdots,s_{2,3}^{(n)})}T_{1,k_2}^{(s_{1,0}^{(n)},\cdots,s_{1,3}^{(n)})},
\end{align*}
where $A=-D+E+G+H-3$\quad,$B=12D+12G-4J-12$,\quad \\
$C=-E-2G-H+4$,\quad$F=-2D-2G-2H+6$,\quad $I=4D-E+2G-H-J$,\quad$t_{1,m}^{(n)}$ is determined in  theorem \ref{cc^n}.
\end{theorem}

Lemma \ref{alg-5} will be discussed in four cases.

Case $1$:
$B=C=D=0.$
\begin{theorem}
\label{tri5-n}
For $m\ge 0$, $n\ge 1$,
 \begin{align*}
&\frac{1}{(A_1^{(n)})^5}\sum_{k_1+\cdots+k_5=m\atop k_1,\cdots,k_5\ge 0}\binom{m}{k_1,\cdots,k_5}
T_{1,k_1}^{(s_{1,0}^{(n)},\cdots,s_{1,3}^{(n)})} \cdots T_{1,k_5}^{(s_{1,0}^{(n)},\cdots,s_{1,3}^{(n)})}\\
&=\frac{A}{A_1^{(5n)}}5^m T_{1,m}^{(s_{1,0}^{(5n)},\cdots,s_{1,3}^{(5n)})}
+\frac{E}{A_1^{(4n)}A_1^n}\sum_{k=0}^m\binom{m}{k}4^{m-k}T_{1,m-k}^{(s_{1,0}^{(4n)},\cdots,s_{1,3}^{(4n)})}
T_{1,k}^{(s_{1,0}^{(n)},\cdots,s_{1,3}^{(n)})}\\
&\quad+F(\frac{-1}{563})^n\sum_{k=0}^m\binom{m}{k}T_{1,k}^{(s_{1,0}^{(n)},\cdots,s_{1,3}^{(n)})}
+\frac{G}{A_1^{(n)}}\sum_{k=0}^m\binom{m}{k}t_{2,k}^{(n)} T_{1,m-k}^{(s_{1,0}^{(n)},\cdots,s_{1,3}^{(n)})}\\
 &\quad+\frac{H}{A_1^{(3n)}A_1^{(2n)}}\sum_{k=0}^m\binom{m}{k}3^{m-k}2^k
T_{1,m-k}^{(s_{1,0}^{(3n)},\cdots,s_{1,3}^{(3n)})}T_{1,k}^{(s_{1,0}^{(2n)},\cdots,s_{1,3}^{(2n)})}\\
&\quad+\frac{I}{A_1^{(3n)}}\sum_{k=0}^m\binom{m}{k}3^{k}T_{1,k}^{(s_{1,0}^{(3n)},\cdots,s_{1,3}^{(3n)})}t_{1,m-k}^{(n)}\\
&\quad+\frac{J}{A_2^{(n)}A_1^{(2n)}}\sum_{k_1+k_2+k_3=m\atop k_1,k_2,k_3\ge0}\binom{m}{k_1,k_2,k_3}
(-1)^{k_1}T_{2,k_1}^{(s_{2,0}^{(n)},\cdots,s_{2,3}^{(n)})}2^{k_2}T_{1,k_2}^{(s_{1,0}^{(2n)},\cdots,s_{1,3}^{(2n)})}\\
\end{align*}
\begin{align*}
&\quad+\frac{K}{A_2^{(n)}}\sum_{k_1+k_2+k_3=m\atop k_1,k_2,k_3\ge0}\binom{m}{k_1,k_2,k_3}(-1)^{k_1}T_{2,k_1}^{(s_{2,0}^{(n)},\cdots,s_{2,3}^{(n)})}t_{1,k_2}^{(n)}\\
&\quad+\frac{L}{A_1^{(3n)}(A_1^{(n)})^2}\sum_{k_1+k_2+k_3=m\atop k_1,k_2,k_3\ge0}\binom{m}{k_1,k_2,k_3}
3^{k_1}T_{1,k_1}^{(s_{1,0}^{(3n)},\cdots,s_{1,3}^{(3n)})}T_{1,k_2}^{(s_{1,0}^{(n)},\cdots,s_{1,3}^{(n)})}
T_{1,k_3}^{(s_{1,0}^{(n)},\cdots,s_{1,3}^{(n)})}\\
&\quad+\frac{M}{A_2^{(n)}(A_1^{(n)})^2}\sum_{k_1+k_2+k_3+k_4=m\atop k_1,k_2,k_3,k_4\ge0}\binom{m}{k_1,k_2,k_3,k_4}
(-1)^{k_1}T_{2,k_1}^{(s_{2,0}^{(n)},\cdots,s_{2,3}^{(n)})}T_{1,k_2}^{(s_{1,0}^{(n)},\cdots,s_{1,3}^{(n)})}
T_{1,k_3}^{(s_{1,0}^{(n)},\cdots,s_{1,3}^{(n)})}\\
&\quad+\frac{N}{(A_1^{(2n)})^2 A_1^n}\sum_{k_1+k_2+k_3=m\atop k_1,k_2,k_3\ge0}\binom{m}{k_1,k_2,k_3}2^{k_1}T_{1,k_1}^{(s_{1,0}^{(2n)},\cdots,s_{1,3}^{(2n)})}
2^{k_2}T_{1,k_2}^{(s_{1,0}^{(2n)},\cdots,s_{1,3}^{(2n)})}T_{1,k_3}^{(s_{1,0}^{(n)},\cdots,s_{1,3}^{(n)})}\\
&\quad+\frac{P}{A_1^{(n)}}\sum_{k_1+k_2+k_3=m\atop k_1,k_2,k_3\ge0}\binom{m}{k_1,k_2,k_3}t_{1,k_1}^{(n)}t_{1,k_2}^{(n)}T_{1,k_3}^{(s_{1,0}^{(n)},\cdots,s_{1,3}^{(n)})}\\
&\quad+\frac{Q}{A_1^{(2n)}A_1^{(n)}}\sum_{k_1+k_2+k_3=m\atop k_1,k_2,k_3\ge 0}\binom{m}{k_1,k_2,k_3}2^{k_1}T_{1,k_1}^{(s_{1,0}^{(2n)},\cdots,s_{1,3}^{(2n)})}
t_{1,k_2}^{(n)} T_{1,k_3}^{(s_{1,0}^{(n)},\cdots,s_{1,3}^{(n)})}\\
&\quad+\frac{R}{A_1^{(2n)} (A_1^{(n)})^3}\sum_{k_1+k_2+k_3+k_4=m\atop k_1,k_2,k_3,k_4\ge 0}\binom{m}{k_1,k_2,k_3,k_4}2^{k_1}T_{1,k_1}^{(s_{1,0}^{(2n)},\cdots,s_{1,3}^{(2n)})}
T_{1,k_2}^{(s_{1,0}^{(n)},\cdots,s_{1,3}^{(n)})}T_{1,k_3}^{(s_{1,0}^{(n)},\cdots,s_{1,3}^{(n)})}
T_{1,k_4}^{(s_{1,0}^{(n)},\cdots,s_{1,3}^{(n)})}
\end{align*}
\begin{align*}
&\quad+\frac{S}{(A_1^{(n)})^3}\sum_{k_1+k_2+k_3+k_4=m\atop k_1,k_2,k_3,k_4\ge 0}\binom{m}{k_1,k_2,k_3,k_4}
t_{1,k_1}^{(n)}T_{1,k_2}^{(s_{1,0}^{(n)},\cdots,s_{1,3}^{(n)})}T_{1,k_3}^{(s_{1,0}^{(n)},\cdots,s_{1,3}^{(n)})}
T_{1,k_4}^{(s_{1,0}^{(n)},\cdots,s_{1,3}^{(n)})},
 \end{align*}
where  \begin{align*}
&A=I+2L+2N+P+2Q+6R+4S-14,\\
& E=-I-2L-N-Q-3R-S+5, \\
&F=4G+I+2L+6N+5P+6Q+18R+16S-50,\\
& H=-L-2N-P-Q-4R-3S+10,\\
&J=-G-I-2L-M-2P-3Q-6R-7S+20, \\
&K=-2G-2M-2N-5P-2Q-6R-12S+30,
\end{align*}
$t_{1,m}^{(n)}$ and  $t_{2,m}^{(n)}$ are determined in  theorem \ref{cc^n} and \ref{tri7}, respectively.
\end{theorem}

Case $2$:
$B\neq 0, $ $C=D=0.$
Let  \begin{align*}
&\sum_{k=0}^\infty t_{3,k}^{(n)} \frac{x^k}{k!}\\
&=c_1^n c_2^n c_3^n e^{(\alpha+\beta+\gamma)x}(c_1^nc_2^n e^{(\alpha+\beta) x}+c_2^n c_3^n e^{(\beta+\gamma) x}
+c_3^n c_1^n e^{(\gamma+\alpha) x})+\cdots\\
&\quad+c_2^n c_3^n c_4^n e^{(\beta+\gamma+\delta)x}(c_2^n c_3^n e^{(\beta+\gamma) x}+c_3^n c_4^n e^{(\gamma+\delta)x}
+c_4^n c_2^n e^{(\delta+\beta) x}).
\end{align*}

\begin{theorem}
\label{tri5-n-2}
For $m\ge 0$, $n\ge 1$,
 \begin{align*}
&Bt_{3,m}^{(n)}
=\frac{1}{(A_1^{(n)})^5}\sum_{k_1+\cdots+k_5=m\atop k_1,\cdots,k_5\ge 0}\binom{m}{k_1,\cdots,k_5}
T_{1,k_1}^{(s_{1,0}^{(n)},\cdots,s_{1,3}^{(n)})} \cdots T_{1,k_5}^{(s_{1,0}^{(n)},\cdots,s_{1,3}^{(n)})}\\
&\quad-\frac{A}{A_1^{(5n)}}5^m T_{1,m}^{(s_{1,0}^{(5n)},\cdots,s_{1,3}^{(5n)})}-\cdots\\
&\quad-\frac{S}{(A_1^{(n)})^3}\sum_{k_1+k_2+k_3+k_4=m\atop k_1,k_2,k_3,k_4\ge 0}\binom{m}{k_1,k_2,k_3,k_4}
t_{1,k_1}^{(n)}T_{1,k_2}^{(s_{1,0}^{(n)},\cdots,s_{1,3}^{(n)})}T_{1,k_3}^{(s_{1,0}^{n},\cdots,s_{1,3}^{(n)})}
T_{1,k_4}^{(s_{1,0}^{(n)},\cdots,s_{1,3}^{(n)})},
 \end{align*}
where  \begin{align*}
&A=-G-J-M+2N-P-Q-3S+6,\\
&B=-2G-K-2M-2N-5P-2Q-6R-12S+30,\\
& E=G+J+M-N+2P+2Q+3R+6S-15,\\
& F=-3G-J-3K-7M-12P-3Q-6R-27S+60, \\
 &H=-L-2N-P-Q-4R-3S+10, \\
 &I=-G-J-2L-M-2P-3Q-6R-7S+20,
  \end{align*}
$t_{1,m}^{(n)}$ and  $t_{2,m}^{(n)}$ are determined in  theorem \ref{cc^n} and \ref{tri7}, respectively.
\end{theorem}

Case $3$:
$C\neq 0, $ $B=D=0.$
Let  \begin{align*}
&\sum_{k=0}^\infty t_{4,k}^{(n)} \frac{x^k}{k!}\\
&=c_1^n c_2^n c_3^n e^{(\alpha+\beta+\gamma)x}(c_1^{2n} e^{2\alpha x}+c_2^{2n} e^{2 \beta x}
+c_3^{2n} e^{2\gamma x})+\cdots\\
&\quad+c_2^n c_3^n c_4^n e^{(\beta+\gamma+\delta)x}(c_2^{2n} e^{2\beta x}+c_3^{2n} e^{2\gamma x}
+c_4^{2n} e^{2\delta x}).
\end{align*}

\begin{theorem}
\label{tri5-n-3}
For $m\ge 0$, $n\ge 1$,
 \begin{align*}
&C t_{4,m}^{(n)}
=\frac{1}{(A_1^{(n)})^5}\sum_{k_1+\cdots+k_5=m\atop k_1,\cdots,k_5\ge 0}\binom{m}{k_1,\cdots,k_5}
T_{1,k_1}^{(s_{1,0}^{(n)},\cdots,s_{1,3}^{(n)})} \cdots T_{1,k_5}^{(s_{1,0}^{(n)},\cdots,s_{1,3}^{(n)})}\\
&\quad-\frac{A}{A_1^{(5n)}}5^m T_{1,m}^{(s_{1,0}^{(5n)},\cdots,s_{1,3}^{(5n)})}-\cdots\\
&\quad-\frac{S}{(A_1^{(n)})^3}\sum_{k_1+k_2+k_3+k_4=m\atop k_1,k_2,k_3,k_4\ge 0}\binom{m}{k_1,k_2,k_3,k_4}
t_{1,k_1}^{(n)}T_{1,k_2}^{(s_{1,0}^{(n)},\cdots,s_{1,3}^{(n)})}T_{1,k_3}^{(s_{1,0}^{(n)},\cdots,s_{1,3}^{(n)})}
T_{1,k_4}^{(s_{1,0}^{(n)},\cdots,s_{1,3}^{(n)})},
 \end{align*}
where  \begin{align*}
&A=I+2L+2N+P+2Q+6R+4S-14,\\
&C=-G-I-J-2L-M-2P-3Q-6R-7S+20,\\
 &E=-I-2L-N-Q-3R-S+5,\\
 & F=3G-J-M+6N+3P+3Q+12R+9S-30,\\
 &H=-L-2N-P-Q-4R-3S+10,\\
  &K=-2G-2M-2N-5P-2Q-6R-12S+30,
  \end{align*}
$t_{1,m}^{(n)}$ and  $t_{2,m}^{(n)}$ are determined in  theorem \ref{cc^n} and \ref{tri7}, respectively.
\end{theorem}

Case $4$:
$D\neq 0, $ $B=C=0.$
Let  \begin{align*}
&\sum_{k=0}^\infty t_{5,k}^{(n)} \frac{x^k}{k!}\\
&=c_1^n c_2^n c_3^n e^{(\alpha+\beta+\gamma)x}(c_1^n e^{\alpha x}+c_2^n e^{\beta x}
+c_3^n e^{\gamma x})^2+\cdots\\
&\quad+c_2^n c_3^n c_4^n e^{(\beta+\gamma+\delta)x}(c_2^n e^{\beta x}+c_3^n e^{\gamma x}
+c_4^n e^{\delta x})^2.
\end{align*}

\begin{theorem}
\label{tri5-n-4}
For $m\ge 0$, $n\ge 1$,
 \begin{align*}
&D t_{5,m}^{(n)}
=\frac{1}{(A_1^{(n)})^5}\sum_{k_1+\cdots+k_5=m\atop k_1,\cdots,k_5\ge 0}\binom{m}{k_1,\cdots,k_5}
T_{1,k_1}^{(s_{1,0}^{(n)},\cdots,s_{1,3}^{(n)})} \cdots T_{1,k_5}^{(s_{1,0}^{(n)},\cdots,s_{1,3}^{(n)})}\\
&\quad-\frac{A}{A_1^{(5n)}}5^m T_{1,m}^{(s_{1,0}^{(5n)},\cdots,s_{1,3}^{(5n)})}-\cdots\\
&\quad-\frac{S}{(A_1^{(n)})^3}\sum_{k_1+k_2+k_3+k_4=m\atop k_1,k_2,k_3,k_4\ge 0}\binom{m}{k_1,k_2,k_3,k_4}
t_{1,k_1}^{(n)}T_{1,k_2}^{(s_{1,0}^{(n)},\cdots,s_{1,3}^{(n)})}T_{1,k_3}^{(s_{1,0}^{(n)},\cdots,s_{1,3}^{(n)})}
T_{1,k_4}^{(s_{1,0}^{(n)},\cdots,s_{1,3}^{(n)})},
 \end{align*}
where  \begin{align*}
&A=I+2L+2N+P+2Q+6R+4S-14,\\
&D=-G-I-J-2L-M-2P-3Q-6R-7S+20,\\
& E=-I-2L-N-Q-3R-S+5,\\
& F=-3G-6I-7J-7M+6N-12L-9P-15Q-24R-33S+90,\\
& H=-L-2N-P-Q-4R-3S+10,\\
& K=2I+2J-2N+4L-P+4Q+6R+2S-10,
  \end{align*}
$t_{1,m}^{(n)}$ and  $t_{2,m}^{(n)}$ are determined in  theorem \ref{cc^n} and \ref{tri7}, respectively.
\end{theorem}

\bigskip

\section{Some more interesting general expressions}

We shall give some more interesting general expressions.
\begin{Lem}
For $n\ge 1$, we have
\begin{align*}
&\quad(c_2 c_3+c_3 c_4+c_4 c_2) e^{\alpha x}+(c_3c_4+c_4 c_1+c_1 c_3) e^{\beta  x}\\
&\quad+(c_1 c_2+c_1 c_4+c_4 c_2) e^{\gamma x}+(c_1 c_2+c_2 c_3+c_1 c_3) e^{\delta x}\\
&=\frac{1}{563}\sum_{k=0}^\infty T_k^{(146,416,581,1080)}\frac{x^k}{k!}.
\end{align*}
\end{Lem}

\begin{theorem}
\label{(cc+cc+cc)^n}
\begin{align*}
&\quad(c_2 c_3+c_3 c_4+c_4 c_2)^n e^{\alpha x}+(c_3c_4+c_4 c_1+c_1 c_3)^n e^{\beta  x}\\
&\quad+(c_1 c_2+c_1 c_4+c_4 c_2)^n e^{\gamma x}+(c_1 c_2+c_2 c_3+c_1 c_3)^n e^{\delta x}\\
&=\frac{1}{A_3^{n}}\sum_{k=0}^\infty T_{3,k}^{(s_{3,0}^{n},s_{3,1}^{n},s_{3,2}^{n},s_{3,3}^{n})}\frac{x^k}{k!}\,,
\end{align*}
where $s_{3,0}^{(n)}$, $s_{3,1}^{(n)}$, $s_{3,2}^{(n)}$, $s_{3,3}^{(n)}$ and $A_3^{(n)}$ satisfy the  recurrence relations:
\begin{align*}
&s_{3,0}^{(n)}=\pm{\rm lcm}(b_1,b_2,b_3),\quad s_{3,1}^{(n)}=Ms_{3,0}^{(n)}, \quad s_{3,2}^{(n)}=Ns_{3,0}^{(n)}, \quad s_{3,3}^{(n)}=Ps_{3,0}^{(n)},\\
&A_3^{(n)}=A_3^{(n-1)}\frac{(-16s_{3,3}^{(n)}+103s_{3,2}^{(n)}-157s_{3,1}^{(n)}-10s_{3,0}^{(n)})}
{-8s_{3,3}^{(n)}+10s_{3,2}^{(n)}+7s_{3,1}^{(n)}-6s_{3,0}^{(n)}}.
\end{align*}
 $b_1$,$b_2$, $b_3$, $M$, $N$ and  $P$ are determined in the proof.
\end{theorem}
\begin{proof}
Similarly to the proof of Theorem \ref{cc^n}, we consider the form
$$
h_1^{(n)}e^{\alpha x}+h_2^{(n)}e^{\beta x}+h_3^{(n)}e^{\gamma x}+h_4^{(n)}e^{\delta x}=\sum_{k=0}^\infty s_{3,k}^{(n)}\frac{x^k}{k!}\,.
$$
By $h_1^{(n)}=A_3^{(n)}(c_2 c_3+c_3 c_4+c_4 c_2)^n$, we can obtain the following recurrence relation:
\begin{align*}
&A=-650s_{3,0}^{(n-1)}+385s_{3,1}^{(n-1)}+854s_{3,2}^{(n-1)}-664s_{3,3}^{(n-1)},\\
& B=1862s_{3,0}^{(n-1)}+231s_{3,1}^{(n-1)}-1627s_{3,2}^{(n-1)}+1178s_{3,3}^{(n-1)},\\
&C=-1100s_{3,0}^{(n-1)}-2380s_{3,1}^{(n-1)}-417s_{3,2}^{(n-1)}+522s_{3,3}^{(n-1)},\\
&D=16s_{3,0}^{(n-1)}-252s_{3,1}^{(n-1)}-170s_{3,2}^{(n-1)}+148s_{3,3}^{(n-1)},\\
&E=1198s_{3,0}^{(n-1)}-1083s_{3,1}^{(n-1)}-1906s_{3,2}^{(n-1)}+1368s_{3,3}^{(n-1)},\\
&F=-2434s_{3,0}^{(n-1)}+814s_{3,1}^{(n-1)}+3473s_{3,2}^{(n-1)}-1769s_{3,3}^{(n-1)},\\
\end{align*}
\begin{align*}
&G=988s_{3,0}^{(n-1)}+1935s_{3,1}^{(n-1)}-757s_{3,2}^{(n-1)}-933s_{3,3}^{(n-1)},\\
&H=-518s_{3,0}^{(n-1)}+801s_{3,1}^{(n-1)}+920s_{3,2}^{(n-1)}-834s_{3,3}^{(n-1)},\\
&I=268s_{3,0}^{(n-1)}+186s_{3,1}^{(n-1)}-132s_{3,2}^{(n-1)}-16s_{3,3}^{(n-1)},\\
&J=-1303s_{3,0}^{(n-1)}-1690s_{3,1}^{(n-1)}+146s_{3,2}^{(n-1)}+666s_{3,3}^{(n-1)},\\
&K=-3980s_{3,0}^{(n-1)}-3273s_{3,1}^{(n-1)}+1638s_{3,2}^{n-1}+620s_{3,3}^{(n-1)},\\
& L=1012s_{3,0}^{(n-1)}-869s_{3,1}^{(n-1)}-1490s_{3,2}^{(n-1)}+1116s_{3,3}^{(n-1)},\\
&M=\frac{(LA-DI)(FA-BE)-(HA-DE)(JA-BI)}{(GA-CE)(JA-BI)-(KA-CI)(FA-BE)},\\
&N=\frac{M(GA-CE)+(HA-DE)}{BE-FA}, \quad P=-\frac{1}{A}(BN+CM+D),
 \end{align*}
 \begin{align*}
 &M=\frac{a_1}{b_1}, \quad N=\frac{a_2}{b_2}, \quad P=\frac{a_3}{b_3},\quad \gcd(a_i,b_i)=1,\\
&s_{3,0}^{(n)}=\pm{\rm lcm}(b_1,b_2,b_3),\quad s_{3,1}^{(n)}=Ms_{3,0}^{(n)}, \quad s_{3,2}^{(n)}=Ns_{3,0}^{(n)}, \quad s_{3,3}^{(n)}=Ps_{3,0}^{(n)},\\
&A_3^{(n)}=A_3^{(n-1)}\frac{(-16s_{3,3}^{(n)}+103s_{3,2}^{(n)}-157s_{3,1}^{(n)}-10s_{3,0}^{(n)})}
{-8s_{3,3}^{(n)}+10s_{3,2}^{(n)}+7s_{3,1}^{(n)}-6s_{3,0}^{(n)}}.
\end{align*}
We choose the symbol of $s_{3,0}^{(n)}$ such that for some $k_0$,$\forall k \geq k_0$,
$T_{3,k}^{(s_{3,0}^{(n)},s_{3,1}^{(n)},s_{3,2}^{(n)},s_{3,3}^{(n)})}$ is positive.
\end{proof}

\end{document}